\documentclass[12pt, a4paper]{amsart}
\usepackage{hyperref}
\usepackage{amsfonts}
\usepackage{amssymb}
\usepackage{amssymb}
\usepackage{enumerate}
\usepackage{latexsym}
\usepackage{nomencl,mathrsfs}
\usepackage{color}
\usepackage{graphicx}
\usepackage{pst-node}
\usepackage{tikz-cd}

\newcommand{\lsp}{\operatorname{span}}

\newcommand{\coker}{\operatorname{coker}}

\newtheorem{theorem}{Theorem}[section]
\newtheorem{lemma}[theorem]{Lemma}
\newtheorem{proposition}[theorem]{Proposition}

\newtheorem{corollary}[theorem]{Corollary}

\newtheorem*{mconj}{Matui's HK Conjecture}

\theoremstyle{definition}
\newtheorem{definition}[theorem]{Definition}
\newtheorem{remark}[theorem]{Remark}
\newtheorem{remarks}[theorem]{Remarks}
\newtheorem{example}[theorem]{Example}
\newtheorem{examples}[theorem]{Examples}
\newtheorem{question}[theorem]{Question}
\numberwithin{equation}{section}

\newcommand{\R}{\mathbb{R}}

\DeclareMathOperator{\Aut}{Aut}

\newcommand{\ima}{\text{im}\,}

\newcommand{\bi}{\begin{itemize}}
\newcommand{\ei}{\end{itemize}}
\newcommand{\be}{\begin{enumerate}}
\newcommand{\ee}{\end{enumerate}}

\newcommand{\N}{\mathbb{N}}
\newcommand{\Z}{\mathbb{Z}}
\newcommand{\Q}{\mathbb{Q}}
\newcommand{\C}{\mathbb{C}}
\newcommand{\T}{\mathbb{T}}
\newcommand{\F}{\mathbb{F}}

\newcommand{\HH}{\mathcal{H}}
\newcommand{\G}{\mathcal{G}}

\newcommand{\M}{\mathfrak{M}}

\newcommand{\id}{\operatorname{id}}

\usepackage[OT2,OT1]{fontenc}
\newcommand\cyr
{
\renewcommand\rmdefault{wncyr}
\renewcommand\sfdefault{wncyss}
\renewcommand\encodingdefault{OT2}
\normalfont
\selectfont
}
\DeclareTextFontCommand{\textcyr}{\cyr}

\title[Ample groupoid homology]{Ample groupoids: Equivalence, Homology, \\ and Matui's HK conjecture}
\author[C. Farsi]{Carla Farsi}
\address[C.F.]{Department of Mathematics, Campus Box 395 \\
Boulder, Colorado 80309-0395, USA}
\email{farsi@colorado.edu}

 \author[A. Kumjian]{Alex Kumjian}
 \address[A.K]{Department of Mathematics (084), University of Nevada\\
 Reno, NV 89557-0084, USA}
 \email{alex@unr.edu}

\author[Pask, Sims]{David Pask and Aidan Sims}
\address[D.P, \&A.S.]{School of Mathematics and Applied Statistics\\
University of Wollongong\\
NSW  2522\\
Australia}
\email{dpask@uow.edu.au, asims@uow.edu.au}

\thanks{C.F.\ was partially supported by Simons Foundation Collaboration grant 523991.
A.K.\ was partially supported  by  Simons Foundation Collaboration grant 353626.
This research was supported by the Australian Research Council.}
\date{\today}

\keywords{Ample groupoid, groupoid equivalence,  homology of \'etale groupoids, K-theory of $C^*$-algebras, Deaconu-Renault groupoid}
\subjclass[2010]{Primary 37B05, 22A22, 46L80, 19D55}

\begin{document}

\begin{abstract}
We investigate the homology of ample Hausdorff groupoids. We establish that a number of
notions of equivalence of groupoids appearing in the literature coincide for ample
Hausdorff groupoids, and deduce that they all preserve groupoid homology. We compute the
homology of a Deaconu--Renault groupoid associated to $k$ pairwise-commuting local
homeomorphisms of a zero-dimensional space, and show that Matui's HK conjecture holds for
such a groupoid when $k$ is one or two. We specialise to $k$-graph groupoids, and show
that their homology can be computed in terms of the adjacency matrices, using a chain
complex developed by Evans. We show that Matui's HK conjecture holds for the groupoids of
single vertex $k$-graphs which satisfy a mild joint-coprimality condition. We also prove
that there is a natural homomorphism from the categorical homology of a $k$-graph to the
homology of its groupoid.
\end{abstract}

\maketitle

\section{Introduction}\label{sec:intro}

The purpose of this paper is to investigate the homology of ample Hausdorff groupoids,
and to investigate Matui's HK-conjecture for groupoids associated to actions of $\N^k$ by
local homeomorphisms on locally compact Hausdorff zero-dimensional spaces. Ample
Hausdorff groupoids are an important source of examples of $C^*$-algebras. They provide
models for the crossed-products associated to Cantor minimal systems \cite{GPS},
Cuntz--Krieger algebras and graph $C^*$-algebras and their higher-rank analogues
\cite{Renault, KPRR, KP2000}, and recently models for large classes of classifiable
$C^*$-algebras \cite{BCSS, DPS, Putnam}. It is therefore very desirable to develop
techniques for computing the $K$-theory of the $C^*$-algebra of an ample Hausdorff
groupoid. Unfortunately, there are relatively few general techniques available.

In a series of recent papers \cite{Matui-2012, Matui-2015, Matui-2016}, Matui has
advanced a conjecture that if $\G$ is a minimal effective ample Hausdorff groupoid with
compact unit space then $K_0(C^*_r(G))$ is isomorphic to the direct sum of the even
homology groups $H_{2n}(\G)$ of the groupoid as defined by Crainic and Moerdijk
\cite{CM}, and $K_1(C^*_r(\G))$ is isomorphic to the direct sum of the odd homology
groups $H_{2n+1}(\G)$. He has verified this conjecture for a number of key classes of
groupoids, including finite cartesian products of groupoids associated to shifts of
finite type, transformation groupoids for Cantor minimal systems, and AF groupoids with
compact unit spaces. He has also developed tools for computing the homology of ample
Hausdorff groupoids, including a spectral sequence that relates the homology of a
groupoid $\G$ endowed with a cocycle $c$ taking values in a discrete abelian group $H$
with the homology of $H$ with values in the homology groups of the skew-product groupoid
$\G \times_c H$. Other authors have subsequently verified Matui's conjecture for
Exel-Pardo groupoids and certain graded ample Hausdorff groupoids (see \cite{ortega}
\cite{HazLi}).

We begin Section~\ref{sec:equivalence} by investigating the many notions of groupoid
equivalence in the literature in the context of arbitrary ample Hausdorff groupoids.
Crainic and Moerdijk focus on the notion of Morita equivalence of groupoids (see
\cite[4.5]{CM}) while Matui employs the notions of similarity (see  \cite[Definition
3.4]{Matui-2012}) and Kakutani equivalence (see  \cite[Definition 4.1]{Matui-2012}).
Similarity of groupoids was previously studied by Renault \cite{Renault} and Ramsay
\cite{Ramsay:AM1971}. In the setting of ample Hausdorff groupoids with $\sigma$-compact
unit spaces, it follows from \cite[Theorem~3.6]{Matui-2012} that Kakutani equivalence
implies similarity. We show that similarity, Kakutani equivalence, Renault's notion of
groupoid equivalence \cite{Ren82, MRW},  and the notion of  groupoid Morita equivalence
of Crainic and Moerdijk (as well as a number of other notions) all coincide for ample
Hausdorff groupoids with $\sigma$-compact unit spaces (see
Theorem~\ref{thm:equivalences}).

In Section~\ref{sec:gpd homology} we recall the definition of homology for an arbitrary
ample Hausdorff groupoid from \cite[3.1]{CM} (see also  \cite[Definition
3.1]{Matui-2012}) and we appeal to a theorem of Crainic and Moerdijk to observe that
groupoid equivalence preserves groupoid homology for arbitrary ample Hausdorff groupoids
(see  \cite[Corollary~4.6]{CM}). Matui also proved that similar Hausdorff \'etale
groupoids have isomorphic homology groups (see  \cite[Proposition 3.5]{Matui-2012}), and
this formulation allows us to give an explicit description of the isomorphism when the
equivalence arises from a similarity. In Section~\ref{sec:HK} we introduce Matui's HK
conjecture, and extend his proof that AF groupoids with compact unit space satisfy the HK
conjecture to the case of non-compact unit spaces.

Our main computations of groupoid homology are in Section~\ref{sec:DR groupoids}, where
we investigate the homology of Deaconu--Renault groupoids $\G(X, \sigma)$ associated to
actions $\sigma$ of $\N^k$ by local homeomorphisms on totally disconnected locally
compact Hausdorff spaces $X$. We adapt techniques developed by Evans \cite{evans} in the
context of $K$-theory for higher-rank graph $C^*$-algebras to construct a chain complex
$A^\sigma$ in which the $n$-chains are elements of $\bigwedge^n \Z^k \otimes C_c(X, \Z)$
and the boundary maps are built from the forward maps $\sigma^n_*$ on $C_c(X, \Z)$ that
satisfy $\sigma^n_*(1_U) = 1_{\sigma^n(U)}$ whenever $U \subseteq X$ is a compact open
set on which $\sigma^n$ is injective. Our main result, Theorem~\ref{thm:main H*
computation}, gives an explicit computation of the homology groups $H_n(\G(X, \sigma))$:
we prove that $H_n(\G(X, \sigma))$ is canonically isomorphic to $H_n(A^\sigma_*)$. We
then show that, if $c : \G(X, \sigma) \to \Z^k$ is the canonical cocycle, then the
homology groups $H_*(A^\sigma)$ also coincide with the homology groups $H_*(\Z^k,
K_0(C^*(\G(X, \sigma) \times_c \Z^k)))$ appearing in Kasparov's spectral sequence for the
double crossed product $C^*(\G(X, \sigma) \times_c \Z^k) \rtimes \Z^k$. Since this double
crossed product is Morita equivalent to $C^*(\G(X, \sigma))$ by Takai duality, this
provides a useful tool for calculating the $K$-theory of $C^*(G(X, \sigma))$. In Theorems
\ref{thm:DR k=1}~and~\ref{thm:DR k=2} we calculate both the $K$-groups of $C^*(G(X,
\sigma))$ and the homology groups of $G(X, \sigma)$ explicitly for $k = 1, 2$, and in
particular prove that ample Deaconu--Renault groupoids of rank at most~2 satisfy Matui's
HK conjecture. We also discuss the differences between Kasparov's spectral sequence and
Matui's for $k \ge 3$ and indicate where one might look for counterexamples to Matui's
conjecture amongst such groupoids.

Finally, in Section~\ref{sec:k-graphs}, we specialise to $k$-graphs. The $k$-graph
groupoid $\G_\Lambda$ of a row-finite $k$-graph $\Lambda$ with no sources is precisely
the Deaconu--Renault groupoid $\G(\Lambda^\infty, \sigma)$ associated to the shift maps
on the infinite-path space of $\Lambda$. We begin the section by linking the homology of
the $k$-graph groupoid with the categorical homology of the $k$-graph by constructing a
.
natural homomorphism from $H_*(\Lambda)$ to $H_*(\G_\Lambda)$. We then investigate how to
apply the results of Section~\ref{sec:DR groupoids} in the specific setting of $k$-graph
groupoids. We prove that the chain complex $A^\sigma_*$ associated to $(\Lambda^\infty,
\sigma)$ as in Section~\ref{sec:DR groupoids} has the same homology as the much simpler
chain complex $D^\Lambda_*$ described by Evans in \cite{evans}. This provides a very
concrete calculation of the homology of a $k$-graph groupoid. It follows that the
homology of $\G_\Lambda$ does not depend on the factorisation rules in $\Lambda$. We use
this and the preceding section to establish an explicit description of the homology of
$1$-graph groupoids and $2$-graph groupoids and to see that these groupoids satisfy
Matui's conjecture. We also prove that for arbitrary $k$, if $\Lambda$ is a $k$-graph
with a single vertex such that the integers $|\Lambda^{e_1}| - 1, \dots, |\Lambda^{e_k}| -
1$ have no nontrivial common divisors, then both the homology of $\G_\Lambda$ and the $K$-theory of
$C^*(\Lambda)$ are trivial, and in particular $\G_\Lambda$ satisfies the HK-conjecture.

\smallskip

\textbf{Acknowledgements:} C.F.\ thanks A.K.\ for his hospitality during her visit to UNR.
A.K.\ thanks his co-authors for their hospitality and support on his visits to Boulder and
Wollongong.

\section{Background}\label{s:definitions}

\subsection{Groupoids and their \texorpdfstring{$C^*$}{C*}-algebras}

We give some brief background on groupoids and their $C^*$-algebras and establish our
notation. For details, see \cite{Exel, Renault, Sims}. A groupoid is a small category
$\G$ with inverses. We write $\G^{(0)}$ for the set of identity morphisms of $\G$, called
the unit space, and we write $r, s : \G \to \G^{(0)}$ for the range and source maps. We
write $\G^{(2)}$ for the set $\{ (\gamma_1 , \gamma_2 ) \in \G \times \G : s ( \gamma_1 )
= r ( \gamma_2 ) \}$ of composable pairs in $\G$. The groupoid $\G$ is a topological
groupoid if it has a locally compact topology under which all operations in $\G$ are
continuous and $\G^{(0)}$ is Hausdorff in the relative topology. If the topology on all
of $\G$ is Hausdorff, we call $\G$ a Hausdorff groupoid. An \'etale groupoid is a
topological groupoid in which $\G^{(0)}$ is open, and $r,s : \G \to \G^{(0)}$ are local
homeomorphisms (in \cite{Renault} such a groupoid is called $r$-discrete with Haar
system). An open subset $U \subseteq \G$ is said to be an \emph{open bisection} if both
$r|_U$ and $s|_U$ are homeomorphisms onto their ranges.
 Given $u \in \G^{(0)}$ we write $\G^u$ for
$\{\gamma \in \G : r(\gamma) = u\}$, $\G_u$ for $\{\gamma \in \G : s(\gamma) = u\}$ and
$\G^u_u = \G^u \cap \G_u$.

A groupoid $\G$ is \emph{ample} if it is \'etale and $\G^{(0)}$ is zero dimensional;
equivalently, $\G$ is ample if it has a basis of compact open bisections.

The \emph{orbit} of a unit $u \in \G^{(0)}$ is the set $[u] := r(\G_u) = s(\G^u)$. A
subset $U \subseteq \G^{(0)}$ is full if $U \cap [u] \not=\emptyset$ for every unit $u$.
We say that $\G$ is \emph{minimal} if the only nontrivial open invariant subset of
$\G^{(0)}$ is $\G^{(0)}$; equivalently, $\G$ is minimal if the closure of $[u]$ is equal to $\G^{(0)}$
for every $u \in \G^{(0)}$.

The \emph{isotropy} of $\G$ is the set $\bigcup_{u \in \G^{(0)}} \G^u_u$ of elements of
$\G$ whose range and source coincide. A groupoid $\G$  is said to be
\emph{effective}\footnote{Matui uses the term \emph{essentially principal}, and the term
\emph{topologically principal} has also been used elsewhere in the literature.} if the
interior of its isotropy coincides with its unit space $\G^{(0)}$

Let $A$ be an abelian group and let $c : \G \to A$ be a $1$-cocycle. Then we may form the
\emph{skew product} groupoid $\G \times_c A$ which is the set $\G \times A$ with
structure maps $r(\gamma, a) = (r(\gamma), a)$, $s(\gamma, a) = (s(\gamma), a +
c(\gamma))$ and $(\gamma, a)(\eta, a+c(\gamma)) = (\gamma\eta, a)$ (see \cite[Definition
I.1.6]{Renault}). There is a natural action $\alpha$ of $A$ on $\G \times_c A$ given by
$\alpha_b( \gamma , a ) = ( \gamma , a+b )$.

Given a Hausdorff \'etale groupoid $\G$, the space $C_c(\G)$ of continuous compactly
supported functions from $\G$ to $\C$ becomes a $*$-algebra with operations given by
\[
    (f*g)(\gamma) = \sum_{\gamma = \gamma_1\gamma_2} f(\gamma_1)g(\gamma_2)\quad\text{ and }\quad f^*(\gamma) = \overline{f(\gamma^{-1})}.
\]
The \emph{groupoid $C^*$-algebra} $C^*(\G)$ is the universal $C^*$-algebra generated by a
$^*$-representation of $C_c(\G)$ (cf.\ \cite[II.1.5]{Renault}). For each unit $u \in
\G^0$ there is a $^*$-representation $\pi_u : C_c(\G) \to \mathcal{B}(\ell^2(\G_u))$
given by $\pi_u(f)\delta_\eta = \sum_{\gamma \in \G_{r(\eta)}} f(\gamma)\delta_{\gamma\eta}$. The reduced
$C^*$-algebra $C^*_r(\G)$ is the closure of the image of $C_c(\G)$ under the
representation $\bigoplus_{u \in \G^{(0)}} \pi_u$.

A cocycle $c : \G \to \Z^k$ determines an action of $\T^k$ by automorphisms of $C_c(\G)$
given by $(z \cdot f)(\gamma) = z^{c(\gamma)}f(\gamma)$, and this extends to an action of $\T^k$ by
automorphisms on each of $C^*(\G)$ and $C^*_r(\G)$. There is an isomorphism $C^*(\G
\times_c \Z^k) \cong C^*(\G) \rtimes \T^k$ that carries a function $f \in C_c(\G \times
\{n\}) \subseteq C_c(\G \times_c \Z^k)$ to the function $z \mapsto (g \mapsto
z^n f(g, n)) \in C ( \T^k , C^* ( \G ) ) \subseteq C^* ( \G ) \rtimes \T^k$. This isomorphism descends to an isomorphism $C_r^*(\G \times_c \Z^k) \cong C_r^*(\G) \rtimes \T^k$.

We will be particularly interested in Deaconu--Renault groupoids associated to actions of
$\N^k$, which are defined as follows. Let $X$ be a locally compact Hausdorff space, and
let $\sigma$ be an action of $\N^k$ on $X$ by surjective local homeomorphisms. The
associated Deaconu--Renault groupoid\footnote{sometimes referred to as a transformation
groupoid (see \cite{exel-renault})} $\G = \G(X, \sigma)$ is defined by
\[
\G = \{(x, p-q, y) \in X \times \Z^k \times X : \sigma^p(x) = \sigma^q(y)\}
\]
(cf.\ \cite{Deaconu-trans,exel-renault}).
We identify $X$ with the unit space via the map $x \mapsto (x, 0, x)$.
The structure maps are given by $r(x, n, y) = x$, $s(x, n, y) = y$ and $(x, m, y) (y, n, z) = (x, m + n, z)$.
A basis for the topology on $\G$ is given by subsets of the form $U \times \{ p-q \} \times V$ where
$U,V$ are open in $X$ and $\sigma^p(U) = \sigma^q(V)$.

There is a natural cocycle $c : \G(X, \sigma) \to \Z^k$ given by $c(x, n, y) := n$.  We
can then form the skew-product groupoid $\G \times_c \Z^k$. With our conventions, in this
groupoid we have $r ( (x,n,y),p)= (x, p)$, $s ((x,n,y),p) = (y, p + n)$ and
$((x,n,y),p)((y,m,z),p+n)=((x,m+n,z),p)$.

There is an action $\widetilde\sigma$ of $\N^k$ on $\widetilde{X} = X \times \Z^k$ by surjective
local homeomorphisms given by $\widetilde\sigma^q (x,p) = (\sigma^q (x), p+q)$.
Moreover there is an isomorphism $\G \times_c \Z^k \cong \G( \widetilde{X}  , \widetilde\sigma)$
given by
\begin{equation}\label{eq:skew groupoid iso}
((x, m, y), p) \mapsto ((x, p), m, (y, p + m)).
\end{equation}

The full and reduced $C^*$-algebras of a Deaconu--Renault groupoid coincide (see for
example \cite[Lemma~3.5]{SW}).

\subsection{\texorpdfstring{$k$}{k}-graphs, their path groupoids, and their \texorpdfstring{$C^*$}{C*}-algebras} \label{sec:graphbasics}

For $k \ge 1$, a \emph{$k$-graph} is a non-empty
countable small category equipped with a functor $d : \Lambda \to \N^k$ that satisfies
the following \emph{factorisation property}. For all $\lambda \in \Lambda$ and $m,n\in
\N^k$ such that $d(\lambda) = m + n$ there exist unique $\mu, \nu\in \Lambda$ such that
$d(\mu) = m$, $d(\nu) = n$, and $\lambda = \mu\nu$. When $d(\lambda) = n$ we say
$\lambda$ has \emph{degree} $n$, and we write $\Lambda^n = d^{-1}(n)$. The standard
generators of $\N^k$ are denoted $\varepsilon_1, \dots ,\varepsilon_k$, and we write $n_i$ for the
$i^{\textrm{th}}$ coordinate of $n \in \N^k$. We define a partial
order on $\N^k$ by $m\leq n$ if $m_i\leq n_i$ for all $i\leq k$.

If $\Lambda$ is a $k$-graph, its \emph{vertices} are the elements of $\Lambda^0$. The
factorisation property implies that these are precisely the identity morphisms, and so
can be identified with the objects. For $\lambda \in \Lambda$ the \emph{source}
$s(\lambda)$ is the domain of $\lambda$, and the \emph{range} $r(\lambda)$ is the
codomain of $\lambda$ (strictly speaking, $s(\lambda)$ and $r(\lambda)$ are the identity
morphisms associated to the domain and codomain of $\lambda$). Given $\lambda,\mu\in
\Lambda$ and $E\subseteq \Lambda$, we define
\begin{gather*}
\lambda E=\{\lambda\nu : \nu\in E, r(\nu)=s(\lambda)\},
\quad E\mu=\{\nu\mu : \nu\in E, s(\nu)=r(\mu)\},\qquad\text{ and}\\
\lambda E\mu=(\lambda E)\mu=\lambda(E\mu).
\end{gather*}
In particular, for $v \in \Lambda^0$ and $n \in \N^k$, we have $v\Lambda^n = \{\lambda
\in \Lambda^n : r(\lambda) = v\}$.

We say that the $k$-graph $\Lambda$ is \textit{row-finite} if $|v{\Lambda}^{n}|<\infty$
 for each $n\in\mathbb{N}^k$ and $v\in{\Lambda}^0$, and has \textit{no sources}
if $0<|v{\Lambda}^{m}|$ for all $v\in{\Lambda}^0$ and $m\in\mathbb{N}^k$.

Let $A$ be an abelian group. Given a functor $c : \Lambda \to A$, we may form the
skew-product $k$-graph $\Lambda \times_c A$ which is the set $\Lambda \times A$ endowed
with structure maps given by $r(\lambda, a) = (r(\lambda), a)$, $s(\lambda, a) =
(s(\lambda), a + c(\lambda))$, $(\lambda, a)(\mu, a+c(\lambda)) = (\lambda\mu, a)$, and
$d ( \lambda, a)= d(\lambda)$ (see \cite[Definition 5.1]{KP2000}). There is a natural
$A$-action $\alpha$ on $\Lambda \times_c A$ given by $\alpha_b( \lambda , a) = ( \lambda
, a+b )$.

\begin{examples}
\begin{enumerate}[(a)]
\item A $1$-graph is the path category of a directed graph (see \cite{KP2000}).
\item Let $\operatorname{Mor} \Omega_k = \{ (m,n) \in \mathbb{N}^k \times \mathbb{N}^k : m \le n \}$, and $\operatorname{Obj} \Omega_k = \mathbb{N}^k$. Define $r,s : \operatorname{Mor} \Omega_k \to
\operatorname{Obj} \Omega_k$ by $r (m,n) = m$,  $s (m,n) = n$,
and for $m \le m \le p \in \N^k$ define $(m,n)(n,p)=(m,p)$ and $d (m,n) = n-m$.
Then $( \Omega_k , d )$ is a $k$-graph. We identify
$\operatorname{Obj} \Omega_k$ with $\{ (m,m) : m \in \mathbb{N}^k \}
\subseteq \operatorname{Mor} \Omega_k$.
\end{enumerate}
\end{examples}

Let $\Lambda$ be a row-finite $k$-graph with no sources. The set $\Lambda^\infty = \{ x :
\Omega_k \to \Lambda \mid x \text{ is a degree-preserving functor}\} $ is called the
\textit{infinite path space} of $\Lambda$. For $v \in \Lambda^0$, we put $v
\Lambda^\infty = \{ x \in \Lambda^\infty : x(0,0)=v \}$.

For $\lambda \in \Lambda$, let $Z ( \lambda ) = \{ x \in \Lambda^\infty : x ( 0 , d (
\lambda ) ) = \lambda \}$. Then $\{ Z ( \lambda ) : \lambda \in \Lambda \}$ forms a basis
of compact open sets for a topology on $\Lambda^\infty$. For $p \in \mathbb{N}^k$, the
shift map $\sigma^p : \Lambda^\infty \to \Lambda^\infty$ defined by
$\big(\sigma^p(x))(m,n) = x(m+p,n+p)$ for $x \in \Lambda^\infty$ is a local homeomorphism
(for more details see \cite[Remark 2.5, Lemma 2.6]{KP2000}).

Following \cite[Definition 2.7]{KP2000} we define the $k$-graph groupoid of $\Lambda$ to
be the Deaconu--Renault groupoid
\begin{equation} \label{eq:dr-kgraph}
\G_\Lambda := \G (\Lambda^\infty, \sigma) =
    \{(x, m-n, y) \in \Lambda^\infty \times \Z^k \times \Lambda^\infty : m,n \in \N^k, \sigma^m(x) = \sigma^n(y)\}.
\end{equation}

The sets $Z(\mu, \nu) := \{(\mu x, d(\mu) - d(\nu), \nu x) : x \in Z(s(\mu))\}$ indexed
by pairs $\mu,\nu \in \Lambda$ with $s(\mu) = s(\nu)$ form a basis of compact open
bisections for a locally compact Hausdorff topology on $\G_\Lambda$. With this topology
$\G_\Lambda$ is an ample Hausdorff groupoid (see \cite[Proposition~2.8]{KP2000}). The
sets $Z(\lambda) = Z(\lambda, \lambda)$ form a basis for the relative topology on
$\G_\Lambda^{(0)} \subseteq \G_\Lambda$. We identify $\G_\Lambda^{(0)} = \{ (x,0,x) : x
\in \Lambda^\infty \}$ with $\Lambda^\infty$. The groupoid $\G_\Lambda$ is minimal if and
only if $\Lambda$ is cofinal \cite[Proof of Proposition~4.8]{KP2000}.

As in the proof of \cite[Theorem~5.2]{KP2000}, there is a bijection between
$\Lambda^\infty \times \Z^k$ and $(\Lambda \times_d \Z^k)^\infty$ given by $(x, p)
\mapsto ((m,n) \mapsto (x(m,n), m+p))$. After making this identification, we obtain an
isomorphism of the groupoid $\G_{\Lambda \times_d \Z^k}$ of the skew-product $k$-graph
with the skew-product groupoid $\G_\Lambda \times_c \Z^k$ corresponding to the canonical
cocycle $c(x, n, y) = n$ via the map $((x,p), m-n, (y,q)) \mapsto ((x, m-n,y), p)$.

Let $\Lambda$ be a row-finite $k$-graph with no sources. A \emph{Cuntz--Krieger
$\Lambda$-family} in a $C^*$-algebra $B$ is a function $s : \lambda \mapsto s_\lambda$
from $\Lambda$ to $B$ such that
\begin{itemize}
\item[(CK1)] $\{s_v : v\in \Lambda^0\}$ is a collection of mutually orthogonal
    projections;
\item[(CK2)] $s_\mu s_\nu = s_{\mu\nu}$ whenever $s(\mu) = r(\nu)$;
\item[(CK3)] $s^*_\lambda s_\lambda = s_{s(\lambda)}$ for all $\lambda\in \Lambda$;
    and
\item[(CK4)] $s_v = \sum_{\lambda \in v\Lambda^n} s_\lambda s_\lambda^*$ for all
    $v\in \Lambda^0$ and $n\in \N^k$.
\end{itemize}
The \emph{$k$-graph $C^*$-algebra $C^*(\Lambda)$} is the universal $C^*$-algebra
generated by a Cuntz--Krieger $\Lambda$-family. There is an isomorphism $C^*
( \Lambda ) \cong C^* ( \G_\Lambda )$ satisfying $s_\lambda \mapsto 1_{Z(\lambda,
s(\lambda))}$ (see \cite[Corollary~3.5(i)]{KP2000}).

As discussed above, we have $\G_{\Lambda \times_d \Z^k} \cong \G_\Lambda
\times_c \Z^k$, and so
\[
C^* ( \G_\Lambda \times_c \Z^k ) \cong C^* ( \G_{\Lambda \times_d \Z^k} ) \cong
C^* ( \Lambda \times_d \Z^k ).
\]
These are approximately finite dimensional (AF) algebras by \cite[Lemma 5.4]{KP2000}.

\subsection{\texorpdfstring{$K$}{K}-theory for \texorpdfstring{$C^*$}{C*}-algebras}

This paper is concerned primarily with calculating groupoid homology, but it is motivated
by the relationship between this and $K$-theory of groupoid $C^*$-algebras, and some of
our key results concern $K$-theory. Readers unfamiliar with $C^*$-algebras and their
$K$-theory, and who are primarily interested in groupoids and groupoid homology will not
need to know more about $C^*$-algebraic $K$-theory than the following points:
$C^*$-algebraic $K$-theory associates two abelian groups $K_0(A)$ and $K_1(A)$ to each
$C^*$-algebra $A$; these groups are invariant under Morita equivalence, and continuous
with respect to inductive limits; the $K_0$-group is the Grothendieck group of a
semigroup consisting of equivalence classes of projections in matrix algebras over $A$;
the $K$-groups of a crossed-product of a $C^*$-algebra $A$ by $\Z$ are related to those
of $A$ by the Pimsner--Voiculescu exact sequence \cite{pv}; and the $K$-groups of a
crossed-product of a $C^*$-algebra $A$ by $\Z^k$ fit into a spectral sequence, due to
Kasparov \cite{kasparov}, involving the homology groups of $\Z^k$ with values in the
$K$-groups of $A$. For more background on $C^*$-algebraic $K$-theory, we refer the
interested reader to \cite{wegge}, \cite{rll} or \cite{B}.

\subsection{\texorpdfstring{$c$}{c}-soft sheaves}\label{subsec:csoft}

Let $X$  be a locally compact Hausdorff space. By a \textit{sheaf of abelian groups}, or
simply a sheaf, over $X$, we mean a (not necessarily Hausdorff) \'etale groupoid
$\mathcal{F}$ with unit space $\mathcal{F}^{(0)} = X$ in which $r = s$, so every element
belongs to the isotropy, and in which each isotropy group $\mathcal{F}_x =
\mathcal{F}^x_x$ is abelian. We think of $\mathcal{F}$ as a group bundle over $X$ with
bundle map  $r = s$. Given a subset $W \subseteq X$, we write $\Gamma(W, \mathcal{F})$
for the set $\{t: W \to \mathcal{F} : t(w) \in \mathcal{F}_w\text{ and }t \text{ is
continuous}\}$  of continuous sections of $\mathcal{F}$ over $W$. A sheaf $\mathcal{F}$
over $X$ is said to be \textit{c-soft} if the restriction map $\Gamma(X, \mathcal{F}) \to
\Gamma(K, \mathcal{F})$ is surjective for any compact set $K \subseteq X$ (see e.g.
\cite[Definition 2.5.5]{KS} or \cite[II.9.1]{Bredon}); that is, if every continuous
section of $\Gamma$ over a compact subset of $X$ extends to a continuous section over all
of $X$.

The property of $c$-softness is a key hypothesis for results of Crainic and Moerdijk (see
\cite{CM}) that we will need in our study of the homology of ample Hausdorff groupoids.
\begin{lemma}
	\label{lemma:extension-sets} Let $X$ be a $0$-dimensional topological space, and let
$K \subseteq X $ be  compact. If $W$ is a compact open subset of $K$ (in the relative
topology on $K$), then there exists a compact open set $V \subseteq X$ such that $W = V
\cap K$.
\end{lemma}
\begin{proof}
Since $W$ is open in the relative topology,  $W= \widehat{V} \cap K,$ with $\widehat{V} $
open in $X$. Let $\mathcal{U}$ be a basis of compact open sets for the topology on $X$;
so $\mathcal{U}' := \{U \in \mathcal{U} : U \subseteq \widehat{V}\}$ satisfies
$\widehat{V}= \bigcup_{U \in \mathcal{U}'} U$. Since $\mathcal{U}'_K:=\{ U \cap K : U \in
\mathcal{U}'\} $ is an open cover of the compact set $W\subseteq K $ (in the relative
topology), There is a finite subset $F \subseteq \mathcal{U}'$ such that
\[
W = \bigcup_{U \in F} (U \cap W).
\]
Now $V := \bigcup_{U \in F} U$ is compact open, and $W = V \cap K$.
\end{proof}

\begin{proposition} \label{prp:constsheafiscsoft}\label{prop:c-soft}
Let $X$ be a $0$-dimensional topological space. Then the constant sheaf $\mathcal{F}$ on
$X$ with values in a discrete abelian group $A$ is c-soft.
\end{proposition}
\begin{proof}
Since $\mathcal{F}$ is the constant sheaf, for every $W \subseteq X$ we have $\Gamma(W,
\mathcal{F}) \cong C(W, A)$. Let $K \subseteq X$ be compact and fix $f \in C(K, A)$. Then
$f(K)$ is a compact subset of the discrete group $A$, and hence finite. For $a \in f(K)$
we let $U_a = f^{-1} (a)$. Since $f$ is continuous each $U_a \subseteq K$ is clopen and
since $K$ is compact $U_a$ is compact and open. By Lemma \ref{lemma:extension-sets}, for
each $a \in A$ there exists a compact open set $V_a \subseteq X$ such that $U_a = V_a
\cap K$. Fix a total order $\le$ on $f(K)$ and for each $a \in f(K)$, define $V'_a := V_a
\backslash \bigcup_{b < a} V_b$. Then each $V_a'$ is a compact open subset of  $X$ since
the $V_a$ are compact open. Moreover, since the $U_a$ are mutually disjoint, we have
$V'_a \cap K = V_a \cap K$ for all $a$. Hence the formula
\[
\tilde{f}(x) :=\left\{  \begin{array}{cc}
a \ &\text{ for } x\in V_a'\\
0\ &\text{ otherwise,}
\end{array}\right.
\]
defines a continuous extension $\tilde{f}$ of  $f$ to $X$.
\end{proof}

\section{Equivalence of ample Hausdorff groupoids}\label{sec:equivalence}

There are a number of notions of equivalence of groupoids that are relevant to us here,
and we need to know that they all coincide for ample Hausdorff groupoids. The notions
that we consider are Morita equivalence \cite{CM}, groupoid equivalence \cite{MRW} (see
also \cite{Ren82}), equivalence via a linking groupoid, equivalence via isomorphic
ampliations, similarity \cite{Ramsay:AM1971, Renault, Matui-2012}, Kakutani equivalence
\cite{Matui-2012} and stable isomorphism \cite{CRS}. We show that the first four of these
notions coincide for all Hausdorff \'etale groupoids (Proposition~\ref{prp:nonample
equiv}), and that they all coincide for ample Hausdorff groupoids with $\sigma$-compact
unit spaces (Theorem~\ref{thm:equivalences}).

The following notion of similarity, recorded by Matui \cite[Definition 3.4]{Matui-2012}
and called homological similarity there, appears in the context of algebraic groupoids in
Renault's thesis \cite{Renault}, and earlier still in \cite{Ramsay:AM1971}---where Ramsay
in turn attributes it to earlier work of Mackey. An adaptation of this notion to the
situation of Lie groupoids also appears in \cite{MoerdijkMrcun} where it is called strong
equivalence.

\begin{definition}[{\cite[Definition~I.1.3]{Renault}, \cite[Definition~3.4]{Matui-2012}}]\label{def:sim}
Let $\mathcal{G, H}$ be Hausdorff \'etale groupoids and let $\rho, \sigma : \G \to
\mathcal{H}$ be continuous homomorphisms. We say that $\rho$ is similar to $\sigma$ if
there is a continuous map   $\theta : \G^{(0)} \to \mathcal{H}$ such that
\[
    \theta(r(g))\rho(g) = \sigma(g)\theta(s(g))\quad\text{ for all $g \in \G$.}
\]
We say that $\G$ and $\mathcal{H}$ are \emph{similar} groupoids if there exist \'etale
homomorphisms $\rho : \G \to \mathcal{H}$ and $\sigma : \mathcal{H} \to \G$ such that
$\sigma \circ \rho$ is similar to $\id_{\G}$ and $\rho \circ \sigma$ is similar to
$\id_{\mathcal{H}}$.   In this case, each of the two maps, $\rho$ and $\sigma$, is called a
\emph{similarity}.
\end{definition}

\begin{remark}\label{rmk:sim equiv}
It is not stated in \cite{Renault} or in \cite{Matui-2012} that similarity of groupoids
is an equivalence relation, but this is standard (it is essentially the argument that
category equivalence is an equivalence relation). It is also easy to give a direct
argument: suppose that $\sigma : \G \to \mathcal{H}$ and $\rho : \mathcal{H} \to \G$
implement a similarity, and that $\alpha : \mathcal{H} \to \mathcal{K}$ and $\beta :
\mathcal{K} \to \mathcal{H}$ also implement a similarity. We aim to show that $\alpha
\circ \sigma$ and $\rho \circ \beta$ implement a similarity. By symmetry it suffices to
find $\kappa : \G^{(0)} \to \G$ such that
\[
\kappa(r(g)) \big(\rho \circ \beta \circ \alpha \circ \sigma\big)(g) = g \kappa(s(g))\quad\text{ for all $g$.}
\]
Since $\rho \circ \sigma \sim \id_{\G}$ and $\beta \circ \alpha \sim \id_{\mathcal{H}}$,
there are $\theta : \G^{(0)} \to \G$ and $\eta : \mathcal{H}^{(0)} \to \mathcal{H}$ such
that
\[
\theta(r(g)) \rho\big(\sigma(g)\big) = g \theta(s(g))
    \quad\text{ and }\quad
\eta(r(h)) \beta \big(\alpha(h)\big) = h \eta(s(h)).
\]
Define $\kappa(x) = \theta(x) \rho \Big(\eta(r(\sigma(x))) \Big)$. Then using that $r(\sigma(r(g)))
= r(\sigma(g))$, and that $\rho$ is a homomorphism, we compute
\begin{align*}
\kappa(r(g))&\big(\rho \circ \beta \circ \alpha \circ \sigma\big)(g)
    = \theta(r(g)) \rho \big(\eta(r(\sigma(g)) \big) \beta\big( \alpha(\sigma(g))) \big) \\
    &= \theta(r(g)) \rho\big( \sigma(g) \big) \rho \big(\eta(s(\sigma(g))) \big)
    = g \theta(s(g)) \rho \big(\eta(s(\sigma(g))) \big)
    = g \kappa(s(g)).
\end{align*}
Thus $(\rho \circ \beta) \circ (\alpha \circ \sigma)$ is similar to $\id_{\G}$ and by symmetry
$(\alpha\circ\sigma)\circ(\rho\circ\beta)$ is similar to $\id_{\mathcal{K}}$.

\end{remark}

\begin{remark}\label{rmk:sim orbit maps}
Let  $\rho, \sigma : \G \to \mathcal{H}$ be continuous groupoid homomorphisms between
Hausdorff \'etale groupoids. Then $\rho$ and $\sigma$ both induce well-defined orbit maps
$[u] \mapsto [\rho(u)]$ and $[u] \mapsto [\sigma(u)]$. Suppose that $\rho$ and $\sigma$
are similar; then $[\rho(u)] = [\sigma(u)]$ for all $u \in \G^{(0)}$, and so the orbit
maps induced by $\rho$ and $\sigma$ are equal. It follows that every similarity of groupoids
induces a bijection between their orbit spaces.
\end{remark}

\begin{definition}[{\cite[Section~4.5]{CM}}]\label{def:weak equiv}
Let $\mathcal{G,H}$ be groupoids. A continuous functor $\varphi : \G \to \mathcal{H}$ is
a \textit{weak equivalence} if
\begin{enumerate}[(i)]
\item the map from $\G^{(0)} *_{\mathcal{H}^{(0)}} \mathcal{H} := \{(u,\gamma) \in
    \G^{(0)} \times \mathcal{H} : \varphi(u) = r(\gamma)\}$ to $\mathcal{H}^{(0)}$
    given by $( u , \gamma ) \mapsto s ( \gamma )$ is an \'etale surjection, and
\item $\G$ is isomorphic to the fibred product
\[
\G^{(0)} *_{\mathcal{H}^{(0)}} \mathcal{H}  *_{\mathcal{H}^{(0)}} \G^{(0)} = \{ (u,\gamma, v) \in \G^{(0)} \times \mathcal{H}  \times \G^{(0)} : r ( \gamma )  = \varphi (u) ,  s ( \gamma) = \varphi (v) \}.
\]
\end{enumerate}
If such a $\varphi$ exists, then we write $\G \stackrel{\sim}{\to} \mathcal{H}$. The
groupoids $\G$ and $\mathcal{H}$ are \emph{Morita equivalent} if there is a groupoid
$\mathcal{K}$ that admits weak equivalences $\G \stackrel{\sim}{\leftarrow} \mathcal{K}
\stackrel{\sim}{\to} \mathcal{H}$.
\end{definition}

\begin{remark}
It is not stated explicitly in \cite{CM} that Morita equivalence is in fact an
equivalence relation, though this is certainly standard. In any case it follows from
Theorem~\ref{thm:equivalences} below.
\end{remark}

In \cite[Proposition 5.11]{MoerdijkMrcun} it is shown that the analogue of similarity for
Lie groupoids implies the analogue of Morita equivalence. We briefly indicate how to
modify their argument to obtain the same result for Hausdorff \'etale groupoids.

\begin{lemma}\label{lem:sim->w.equiv}
Let $\mathcal{G,H}$ be Hausdorff \'etale groupoids. If $\G$ and $\mathcal{H}$ are
similar, then they are Morita equivalent.
\end{lemma}
\begin{proof}
Suppose that $\sigma : \G \to \mathcal{H}$ and $\rho : \mathcal{H} \to \G$ constitute a
similarity of groupoids. We claim that the map $\sigma$ is a weak equivalence. It is easy
to show that $(u, h) \mapsto s(h)$ is a surjection from $\G^{(0)} *_{\G^{(0)}}
\mathcal{H}$ to $\mathcal{H}^{(0)}$ since any $v \in \HH^{(0)}$ is the image of $( \rho
(v) , v )$. We claim that it is an \'etale map. For this, fix $(u,h) \in \G^{(0)}
*_{\G^{(0)}} \mathcal{H}$, and use that $\sigma$ is an \'etale map to choose a
neighbourhood $U$ of $u$ such that $\sigma|_U$ is a homeomorphism onto its range. Pick a
bisection neighbourhood $B$ of $h$. By shrinking if necessary, we can assume that
$\sigma(U) = r_{\mathcal{H}}(B)$. Then $s \circ \pi_2$ is a homeomorphism on $U * B \cong
B$.

Since $\rho , \sigma$ constitute a similarity of groupoids it is straightforward to see
that the map $g \mapsto (r(g), \sigma(g), s(g))$ is a bijection from $\G$ to $\G^{(0)}
*_{\mathcal{H}^{(0)}} \mathcal{H} *_{\mathcal{H}^{(0)}} \G^{(0)}$. For $g \in G$ we use
that $r, \sigma$, and $s$ are \'etale maps to find a neighbourhood $U$ of $g$ on which
they are all homeomorphisms, and observe that then $g \mapsto (r(g), \sigma(g), s(g))$ is
a homeomorphism onto $r(U) *_{\mathcal{H}^{(0)}} \sigma(U) *_{\mathcal{H}^{(0)}} s(U)$.
So $g \mapsto (r(g), \sigma(g), s(g))$ is continuous and open.
\end{proof}

The third notion of equivalence that we consider is the one formulated by Renault (see
\cite[Section~3]{Ren82}) and studied in \cite{MRW}. Given a locally compact Hausdorff
groupoid $\G$, we say that a locally compact Hausdorff space $Z$ is a \emph{left
$\G$-space} if it is equipped with a continuous open map $r : Z \to \G^{(0)}$ and a
continuous pairing $(g, z) \mapsto g \cdot z$ from $\G * X$ to $X$ such that $r(g \cdot
z) = r(g)$ and $(gh) \cdot z = g \cdot (h \cdot z)$ and such that $r(z) \cdot z = z$. We
say that $Z$ is a free and proper left $\G$-space if the map $(g, x) \mapsto (g\cdot x,
x)$ is a proper injection from $\G * X$ to $\G \times X$. Right $\G$-spaces are defined
analogously.

\begin{definition}[{\cite[Definition~2.1]{MRW},\cite[\S3]{Ren82}}]
The groupoids $\G$ and $\mathcal{H}$ are \emph{equivalent} if there is a locally compact
Hausdorff space $Z$ such that
\begin{enumerate}
\item $Z$ is a free and proper left $\G$-space with fibre map $r : Z \to \G^{(0)}$,
\item $Z$ is a free and proper right $\mathcal{H}$-space with fibre map $s : Z \to
    \mathcal{H}^{(0)}$,
\item the actions of $\G$ and $\mathcal{H}$ on $Z$ commute,
\item $r : Z \to \G^{(0)}$ induces a homeomorphism $Z/\mathcal{H} \to \G^{(0)}$, and
\item $r : Z \to \mathcal{H}^{(0)}$ induces a homeomorphism $\G\backslash Z \to
    \mathcal{H}^{(0)}$.
\end{enumerate}
\end{definition}

The fourth notion of equivalence we need to discuss is the generalisation of Kakutani
equivalence developed by Matui \cite[Definition~4.1]{Matui-2012} in the situation of
ample Hausdorff groupoids with compact unit spaces, and extended to non-compact unit
spaces in \cite{CRS}. This notion has previously been discussed only for ample groupoids,
but it makes sense for general Hausdorff \'etale groupoids, and in particular weak
Kakutani equivalence is a fairly natural notion in this setting (though in this more general setting
it is not an equivalence relation, see Example~\ref{ex:sim ne equiv} ).

\begin{definition}
The Hausdorff \'etale groupoids $\G$ and $\mathcal{H}$ are \emph{weakly Kakutani
equivalent} if there are full open subsets $X \subseteq \G^{(0)}$ and $Y \subseteq
\mathcal{H}^{(0)}$ such that $\G|_X \cong \mathcal{H}|_Y$. They are \emph{Kakutani
equivalent} if $X$ and $Y$ can be chosen to be clopen sets.
\end{definition}

For ample Hausdorff groupoids with $\sigma$-compact unit spaces, \cite[Theorem~3.2]{CRS}
shows that weak Kakutani equivalence and Kakutani equivalence both coincide with groupoid
equivalence, and with a number of other notions of equivalence. Our next two results show
first that for Hausdorff \'etale groupoids, Morita equivalence and equivalence in the
sense of \cite{MRW} are equivalent to the existence of a linking groupoid, and to
existence of isomorphic ampliations of the two groupoids in the following sense.

If $\G$ is a Hausdorff \'etale groupoid, $X$ is a locally compact Hausdorff space, and
$\psi : X \to \G^{(0)}$ is a local homeomorphism, then the \emph{ampliation} (also known
as the \emph{blow-up} \cite[\S3.3]{Wil16}) $\G^\psi$ of $\G$  corresponding to $\psi$ is
given by
\[
\G^{\psi} = \{(x, \gamma, y) \in X \times \G \times X : \psi(x) = r(\gamma)\text{ and } \psi(y) = s(\gamma)\}
\]
with $\big((x, \gamma, y), (w, \eta, z)\big) \in (\G^{\psi})^{(2)}$ if and only if $y =
w$, and composition and inverses given by $(x, \gamma, y)(y, \eta, z) = (x, \gamma\eta,
z)$ and $(x, \gamma, y)^{-1} = (y, \gamma^{-1}, x)$. This is a Hausdorff \'etale groupoid
under the relative topology inherited from $X \times \G \times X$.

\begin{example}  \label{ex:rpsidef}
Let $X$ and $Y$ be locally compact Hausdorff spaces and $\psi: Y \to X$ be a local homeomorphism.
Then we may regard
\[
R(\psi) := \{ (y_1, y_2) \in Y \times Y : \psi(y_1) = \psi(y_2) \}
\]
as an Hausdorff \'etale groupoid (see \cite{kumjian}).  Note that $R(\psi)$ is the
ampliation of the trivial groupoid $X$ corresponding to $\psi$.
\end{example}

\begin{proposition}\label{prp:nonample equiv}
Let $\G$ and $\mathcal{H}$ be Hausdorff \'etale groupoids. The following are equivalent:
\begin{enumerate}
\item $\G$ and $\mathcal{H}$ are Morita equivalent;
\item there is a Hausdorff \'etale groupoid $\mathcal{L}$ and a decomposition
    $\mathcal{L}^{(0)} = X \sqcup Y$ of $\mathcal{L}^{(0)}$ into complementary full
    clopen subsets such that $\mathcal{L}|_X \cong \G$ and $\mathcal{L}|_Y \cong
    \mathcal{H}$;
\item $\G$ and $\mathcal{H}$ are equivalent in the sense of Renault; and
\item $\G$ and $\mathcal{H}$ admit isomorphic ampliations.
\end{enumerate}
If $\G$ and $\HH$ are weakly Kakutani equivalent, then they satisfy (1)--(4).
\end{proposition}
\begin{proof}
For \mbox{(1)\;$\implies$\;(2)}, suppose that $\phi : \G \to \mathcal{H}$ is a weak
equivalence in the sense of Definition~\ref{def:weak equiv}. Then $\G$ is isomorphic to
the fibred product
\[
\G^{(0)} * \mathcal{H} * \mathcal{\G}^{(0)}
    := \{(x, \eta, y) \in \G^{(0)} \times \mathcal{H} \times \mathcal{\G}^{(0)} : \phi(x) = r(\eta)\text{ and }\phi(y) = s(\eta)\}.
\]
Using the first condition of Definition~\ref{def:weak equiv} it is straightforward to check that under the natural operations and topology, the disjoint union
\[
\mathcal{L} :=
    (\G^{(0)} * \mathcal{H} * \mathcal{\G}^{(0)}) \sqcup
    (\G^{(0)} * \mathcal{H}) \sqcup
    (\mathcal{H} * \mathcal{\G}^{(0)}) \sqcup
    \mathcal{H}
\]
is a Hausdorff \'etale groupoid satisfying~(2) with respect to $X = \G^{(0)}$ and $Y :=
\mathcal{H}^{(0)}$.

For \mbox{(2)\;$\implies$\;(3)}, one checks that given $\mathcal{L}$ as in~(2), the
subspace
\[
Z := \{z \in \mathcal{L} : r(z) \in X\text{ and } s(z) \in Y\},
\]
under the actions of $\G$ and $\mathcal{H}$ by multiplication on either side, is a
$\G$--$\mathcal{H}$-equivalence as in~(3).

For \mbox{(3)\;$\implies$\;(4)}, suppose that $Z$ is a $\G$--$\mathcal{H}$-equivalence;
to avoid confusion, we will write $\rho : Z \to \G^{(0)}$ and $\sigma : Z \to
\mathcal{H}^{(0)}$ for the anchor maps. Since $Z/\mathcal{H} \cong \G^{(0)}$ and since
the right $\mathcal{H}$-action is free, if $x,y \in Z$ satisfy $\rho(x) = \rho(y)$, then
there is a unique element $[x, y]_{\mathcal{H}}$ of $\mathcal{H}$ satisfying $x \cdot [x,
y]_{\mathcal{H}} = y$. By \cite[\S{2}]{MRW} the map $[\cdot, \cdot]_{\mathcal{H}}$ is
continuous (see also \cite[Lemma~2.1]{SW1}).   Similarly there is a continuous pairing
$(x,y) \mapsto {_\G[x, y]}$ from $\{(x,y) \in Z^2 : \sigma(x) = \sigma(y)\}$ to $\G$ such
that ${_{\G}[x,y]} \cdot y = x$. Consider the ampliations
\begin{align*}
\G^{\rho} &:= \{(x,\gamma, y) : x,y \in Z, \gamma \in \G, \rho(x) = r(\gamma), \rho(y) = s(\gamma)\}\quad\text{ and}\\
\mathcal{H}^{\sigma} &:= \{(x,\eta, y) : x,y \in Z, \eta \in \mathcal{H}, \sigma(x) = r(\eta), \sigma(y) = s(\eta)\}.
\end{align*}
If $(x, \gamma, y) \in \G^{\rho}$, then $\rho(\gamma \cdot y) = r(\gamma) = \rho(x)$, and
so we can take the pairing $[x, \gamma \cdot y]_{\mathcal{H}}$ to obtain an element
$\Theta(x, \gamma, y) := (x, [x, \gamma \cdot y]_{\mathcal{H}}, y) \in
\mathcal{H}^\sigma$. It is routine to check that this is a continuous groupoid
homomorphism. Symmetrically, we see that $\Theta' : (x, \eta, y) \mapsto (x, {_\G[x\cdot
\eta, y]}, y)$ is a continuous groupoid homomorphism from $\mathcal{H}^\sigma$ to
$\G^\rho$. A simple calculation using the defining properties of ${_{\G}[\cdot,\cdot]}$
and $[\cdot,\cdot]_{\mathcal{H}}$ shows $\Theta$ and $\Theta'$ are mutually inverse. So
$\Theta$ is an isomorphism, giving~(4).

For \mbox{(4)\;$\implies$\;(1)}, fix ampliations $\G^\phi$ and $\mathcal{H}^\psi$ and an
isomorphism $\Theta : \G^\phi \to \mathcal{H}^\psi$. Write $\pi^\G$ for the canonical map
$(x, \gamma, y) \mapsto \gamma$ from $\G^\phi$ to $\G$, and $\pi^{\mathcal{H}}$ for the
corresponding map from $\mathcal{H}^\psi$ to $\mathcal{H}$. We obtain continuous groupoid
homomorphisms $\tilde{\phi} : \G^\phi \to \G$ and $\tilde{\psi} : \G^\phi \to
\mathcal{H}$ by $\tilde{\phi} := \pi^\G$ and $\tilde{\psi} := \pi^{\mathcal{H}} \circ
\Theta$. It is routine to check that this determines a Morita equivalence $\G
\stackrel{\tilde{\phi}}{\leftarrow} \G^\phi \stackrel{\tilde{\psi}}{\to} \mathcal{H}$.

For the final statement, observe that if $U$ is a full open subset of $\G^0$, then the
argument of \cite[Lemma~6.1]{CS} shows that $\G U = \{g \in \G : s(g) \in U\}$ is a
$\G$--$\G|_U$-equivalence under the actions determined by multiplication in $\G$. So,
writing $\sim_R$ for equivalence in the sense of Renault, if $\G$ and $\HH$ are weakly
Kakutani equivalent, say $\G|_U \cong \HH|_V$, then we have $\G \sim_R \G|_U \cong \HH|_V
\sim_R \HH$. Since $\sim_R$ is an equivalence relation, we deduce that $\G \sim_R \HH$.
\end{proof}

\begin{remarks} \label{rmk:weakeq}
\begin{enumerate}[(i)]
\item Recall the definition of  $R ( \psi )$ from Example \ref{ex:rpsidef}.
    Proposition~\ref{prp:nonample equiv} shows that $R ( \psi )$ and $X$ are
    equivalent, and so $C^*(R(\psi))$ is Morita equivalent to $C_0(X)$.
\item It follows from the proof of Proposition~\ref{prp:nonample equiv} that if $\G$
    and $\mathcal{H}$ admit isomorphic ampliations, then there exist a locally
    compact Hausdorff space $X$, local homeomorphisms $\phi : X \to \G^{(0)}$ and
    $\psi : X \to \mathcal{H}^{(0)}$, and an isomorphism $\Theta : \G^\phi \to
    \mathcal{H}^\psi$ such that $\Theta(x, \phi(x), x) = (x, \psi(x), x)$ for all $x
    \in X$.
 \item Let $\G$ and $\mathcal{H}$ be minimal Hausdorff \'etale groupoids which are
     equivalent in the sense of Renault. Then with notation as in the
     Proposition~\ref{prp:nonample equiv}(2),  we may identify $\G = \mathcal{L}|_X$
     and $\mathcal{H} = \mathcal{L}|_Y$  where $\mathcal{L}$  is a Hausdorff \'etale
     groupoid and $X$ and $Y$ are complementary full clopen subsets  of
     $\mathcal{L}^{(0)}$.
      Let $U \subseteq \mathcal{L}$ be an open bisection such that
    $r(U) \subseteq  X$ and $s(U) \subseteq Y$.
    Since $\mathcal{L}$ is minimal, both $r(U)$ and $s(U)$ are full open subsets.
    It follows that $\G|_{r(U)} \cong \mathcal{H}|_{s(U)}$ and so $\G$ and $\mathcal{H}$ are weakly Kakutani equivalent.
\end{enumerate}
\end{remarks}

Our next result shows that the notions of equivalence in Proposition~\ref{prp:nonample
equiv} are further equivalent to a number of additional conditions, including similarity,
in the special case of ample Hausdorff groupoids with $\sigma$-compact unit spaces. We
write $\mathcal{R}$ for the (discrete) full equivalence relation $\mathcal{R} = \N \times
\N$.

\begin{theorem}\label{thm:equivalences}
Let $\G$ and $\mathcal{H}$ be ample Hausdorff groupoids with $\sigma$-compact unit
spaces. Then the following are equivalent:
\begin{enumerate}
\item $\G$ and $\mathcal{H}$ are similar;
\item $\G$ and $\mathcal{H}$ are Morita equivalent;
\item there exist an ample Hausdorff groupoid $\mathcal{L}$ and a decomposition
    $\mathcal{L}^{(0)} = X \sqcup Y$ of $\mathcal{L}^{(0)}$ into complementary full
    clopen subsets such that $\mathcal{L}|_X \cong \G$ and $\mathcal{L}|_Y \cong
    \mathcal{H}$;
\item $\G$ and $\mathcal{H}$ are equivalent in the sense of Renault;
\item $\G$ and $\mathcal{H}$ admit isomorphic ampliations;
\item $\G \times \mathcal{R} \cong \mathcal{H} \times \mathcal{R}$;
\item $\G$ and $\mathcal{H}$ are Kakutani equivalent; and
\item $\G$ and $\mathcal{H}$ are weakly Kakutani equivalent.
\end{enumerate}
\end{theorem}
\begin{proof}
\mbox{(1)\;$\implies$\;(2)} follows from Lemma~\ref{lem:sim->w.equiv}.

Proposition~\ref{prp:nonample equiv} shows that (2)--(5) are equivalent, and
\cite[Theorem~3.2]{CRS} shows that(4), (6), (7)~and~(8) are equivalent. In particular, we
have (1)\;$\implies$\;(2)\;$\implies$\;$\cdots$\;$\implies$\;(8).

For \mbox{(8)\;$\implies$\;(1)}, suppose that $U \subseteq \G^{(0)}$ and $V \subseteq
\mathcal{H}^{(0)}$ are full open sets with $\G|_U \cong \mathcal{H}|_V$. Matui proves in
\cite[Theorem~3.6(2)]{Matui-2012} that $\G$ is similar to $\G|_U$ and $\mathcal{H}$ is
similar to $\mathcal{H}|_V$. Since any isomorphism of groupoids is a similarity, and
since similarity of groupoids is an equivalence relation (Remark~\ref{rmk:sim equiv}), it
follows that $\G$ and $\mathcal{H}$ are similar.
\end{proof}

To close the section, we present an example to show that groupoid equivalence does not
imply either similarity or weak Kakutani equivalence in general.
We also show that weak Kakutani equivalence is not an equivalence relation.

\begin{example}\label{ex:sim ne equiv}
Let $Y := \R$ and $X := S^1$ and define $\psi: Y \to X$ by $\psi(y) = e^{2{\pi}iy}$. Then
$\psi$ is a local homeomorphism and the groupoid $R(\psi)$ (see Example~\ref{ex:rpsidef})
is equivalent to the trivial groupoid $X$ (see Remarks~\ref{rmk:weakeq}(1)).

We claim that there is no similarity $\rho: X \to R(\psi)$. Indeed suppose that such a
$\rho$ exists. Since $\rho$ is a groupoid map, we have $\rho(X^{(0)}) \subseteq
R(\psi)^{(0)}$. Identifying $R(\psi)^{(0)} = \R$ and $X^{(0)} = S^1$ we obtain a
continuous map $\rho: S^1 \to \R$. Since $\rho$ is a similarity it induces a bijective
map on orbits (see Remark~\ref{rmk:sim orbit maps}). Since the orbits in $X$ are
singletons, this implies that $\rho$ is injective which is impossible.

We also claim that $X$ and $R(\psi)$ are not weakly Kakutani equivalent. To see this,
suppose that $U$ is a full open subset of $X^{(0)}$. Since $X$ is trivial, we have $U =
X^{(0)} = S^1$, which is not homeomorphic to any open subset of $\R = R(\psi)^{(0)}$. So
there is no full open subset $V \subseteq R(\psi)^{(0)}$ such that $X|_U \cong R(\psi)|_V$, and so
the two groupoids are not weakly Kakutani equivalent.

Consider the local homeomorphism $\varphi: X \sqcup Y \to X$ given by
\[
\varphi(z) :=
\begin{cases}
 z & \text{if  } z \in X, \\
 \psi(z) &  \text{if  } z \in Y.
\end{cases}
\]
Then $X$  and $R(\psi)$ are each weakly Kakutani equivalent to
$R(\varphi)$ but as shown above $X$ is not weakly Kakutani equivalent to $R(\psi)$.
\end{example}

\section{Crainic--Moerdijk--Matui homology for ample Hausdorff groupoids}\label{sec:gpd homology}

Crainic and Moerdijk introduced a compactly supported homology theory for Hausdorff
\'etale groupoids in \cite{CM}. Matui reframed the theory for ample Hausdorff groupoids
(though he did not explicitly require this; see \cite[Definition~3.1]{Matui-2012}). To
use the results of \cite{CM} we must ensure that the standing assumptions of
\cite[Section~2.5]{CM}) are satisfied. We therefore require that all groupoids we
consider henceforth are locally compact, Hausdorff, second countable, and zero
dimensional.

For the reader's convenience we recall Matui's definition of homology for an ample
Hausdorff groupoid $\G$ (see \cite[Section~3.1]{Matui-2012}). Since a locally constant
sheaf over such a groupoid with values in a discrete abelian group is $c$-soft (see
Section~\ref{subsec:csoft}), this agrees with the definition given by Crainic and
Moerdijk \cite[Section~3.1]{CM} under our standing assumptions.

We first need to establish some notation. Given a locally compact Hausdorff
zero-dimensional space $X$ and a discrete abelian group $A$, let $C_c(X, A)$ denote the
set of compactly supported $A$-valued continuous (equivalently, locally constant)
functions on $X$. Then $C_c(X, A)$ is an abelian group under pointwise addition. Given a
(not necessarily surjective) local homeomorphism $\psi: Y \to X$ between two such spaces,
as in \cite[Section~3.1]{Matui-2012} we define a homomorphism $\psi_*: C_c(Y, A) \to
C_c(X, A)$ by
\begin{equation}\label{eq:homology induced map}
\psi_*(f)(x) := \sum_{\psi(y) = x} f(y) \quad\text{ for all } f \in C_c(Y, A), \text{ and } x \in X.
\end{equation}
If $U \subseteq Y$ is compact open and $\psi|_U$ is injective, then
$\psi_*(1_U) = 1_{\psi(U)}$, where $1_U$ is the indicator function of $U$.

Recall that for $n > 0$, the space of composable $n$-tuples in a groupoid $\G$ is
\begin{equation} \label{eq:gndef}
\G^{(n)} = \{ (g_1, \dots, g_n) \in \G^n : s(g_i) = r(g_{i+1}) \ \text{for } 1 \le i < n \},
\end{equation}
while $\G^{(0)}$ is the unit space.  For $n \ge 2$ and $0 \le i \le n$ we define $d_i :
\G^{(n)} \to \G^{(n-1)}$ by
\[
d_i(g_1, \dots, g_n) :=
\begin{cases}
(g_2, \dots, g_n) & i = 0, \\
(g_1, \dots, g_ig_{i+1}, \dots, g_n) & 1 \le i \le n-1, \\
(g_1, \dots, g_{n-1}) & i = n.
\end{cases}
\]
Note that $\G^{(n)}$ is $0$-dimensional and each $d_i$ is a local homeomorphism.

\begin{definition}\label{dfn:gpd homology}
Let $\G$ be a second-countable ample Hausdorff groupoid. For $n \ge 1$ define $\partial_n:
C_c(\G^{(n)}, A) \to C_c(\G^{(n-1)}, A)$ by
\[
\partial_1 = s_* - r_*\qquad\text{ and }\qquad\partial_n := \sum_{i=0}^n (-1)^i(d_i)_* \quad \text{for $n \ge 2$},
\]
and define $\partial_0$ to be the zero map from $C_c(\G^{(0)}, A)$ to $0$. Routine
calculations show that this defines a chain
complex $(C_c(\G^{(*)}, A), \partial_*)$. We define the homology of $\G$ with values in
$A$ to be the homology of this complex, denoted $H_*(\G, A)$. If $A = \Z$ we simply write
$H_*(\G)$.
\end{definition}

\begin{remark}
An ample groupoid with one unit is just a  discrete group. In this instance the groupoid
homology just defined coincides with group homology, see
\cite[Section~2.22]{crainic-general-case}, and also \cite{brown}.
\end{remark}

Matui shows in \cite[Proposition 3.5]{Matui-2012} that if  $\G$ and $\HH$ are similar,
then $H_*(\G, A) \cong H_*(\HH, A)$ for any discrete abelian group $A$.  So
Theorem~\ref{thm:equivalences} implies that if $\G$ and $\HH$ are ample and have
$\sigma$-compact unit spaces and are equivalent via any of the eight notions of
equivalence listed in the statement of the theorem, then their homologies coincide. We
will also an explicit description of the isomorphism.

\begin{lemma}\label{lem:morita-invariance}
Let $\G$ and $\mathcal{H}$ be ample Hausdorff groupoids with $\sigma$-compact unit
spaces. If the pair $\G$ and $\mathcal{H}$ satisfies any of the eight equivalent
conditions in Theorem~\ref{thm:equivalences}, then $H_n(\G) \cong H_n(\HH)$. In
particular, if $\rho: \G \to \mathcal{H}$ is a similarity (see Definition \ref{def:sim})
then it induces an isomorphism $\rho_*: H_n(\G) \cong H_n(\HH)$. If $X$ is a full open
subset of $\G^{(0)}$ then the inclusion $\G |_X \subseteq \G$ is a similarity and induces
an isomorphism $H_* (\G|_X) \cong H_* (\G)$.
\end{lemma}
\begin{proof}
Crainic and Moerdijk show that a Morita equivalence between Hausdorff \'etale groupoids
induces an isomorphism between their homology groups  (see \cite[Corollary 4.6]{CM}). The
proof of \cite[Proposition 3.5]{Matui-2012} shows that if $\rho:  \G \to \mathcal{H}$ is
a similarity then $\rho$ induces an isomorphism $H_n ( \G ) \cong H_n ( \HH)$ for all $n
\ge 0$. Let $X$ be a full open subset of $\G^{(0)}$ then the argument of \cite[Theorem
3.6]{Matui-2012} proves that the inclusion $\G |_X \subseteq \G$ is a similarity.  Hence,
the inclusion map induces an isomorphism $H_* (\G|_X ) \cong H_* ( \G  )$.
\end{proof}

\begin{proposition}\label{prop:trivial-groupoid}
Let $X$ be a 0-dimensional space.  If we regard  $X$ as an ample groupoid with $X^{(0)} =
X$ and with trivial multiplication then
\[
H_n(X) =
\begin{cases}
C_c(X, \Z) & \text{if } n = 0, \\
0 & \text{otherwise.}
\end{cases}
\]
\end{proposition}
\begin{proof} The boundary maps for the groupoid $X$ are all trivial and there are no
nondegenerate $n$-chains for $n \ge 1$.
\end{proof}

\begin{remark}\label{rmk:trivial-case}
\begin{enumerate}
\item Following \cite[Definition 3.1]{Matui-2012}, if $\G$ is an ample Hausdorff
    groupoid, then there is a natural preorder on $H_0(\G)$ determined by the cone
    \[
    H_0(\G)^+ := \{ [f] : f \in C_c(\G^{(0)}, \Z) \text{ and } f(x) \ge 0 \text{ for all }x \in \G^{(0)} \}.
    \]
\item For any $0$-dimensional space $X$ regarded as a groupoid as in Proposition~\ref{prop:trivial-groupoid}, we have $C^*(X) \cong C_0(X)$, and
    \[
    K_0(C_0(X)) \cong H_0(X) \cong C_c(X, \Z) ,
    \]
via an isomorphism that carries the positive cone of  $K_0(C_0(X))$  to $H_0(X)^+$.
\item The notion of a type semigroup for the transformation group $(X ,\Gamma), $
    where $X$ is a Cantor set and $\Gamma $ is discrete was introduced in
    \cite{SieraRor}. This idea was generalised by Rainone and Sims in
    \cite[Definition~5.4]{RaiSims}, and independently by B{\"o}nicke and Li
    \cite{BonLi}, who introduced the type semigroup $S(\G)$ of an ample Hausdorff
    groupoid $\G$. The map $[ 1_U ]_{H_0 ( \G )} \mapsto [ 1_U ]_{G(S(\G))}$ induces
    an isomorphism of the homology group $H_0(\G)$ onto the Grothendieck group of
    $S(\G)$. This isomorphism carries $H_0(\G)^+$ to the image of $S(\G)$ in its
    Grothendieck group. In particular, the coboundary subgroup $H_\G$ of
 \cite[Definition~6.4]{RaiSims} is exactly $\ima{\partial_1}$ as defined above.
\end{enumerate}
\end{remark}

\begin{remark}\label{rmk:functoriality}
Homology for ample Hausdorff groupoids is functorial in the following sense. Let $\G$ and
$\HH$ be ample Hausdorff groupoids and let $\phi: \G \to \HH$ be an \'etale groupoid
homomorphism (so in particular, $\phi$ is a local homeomorphism). Then as Crainic and
Moerdijk observe (see \cite[3.7.2]{CM}), the maps $\phi^{(n)}_*: C_c(\G^{(n)}, \Z) \to
C_c(\HH^{(n)}, \Z)$ induce homomorphisms on homology which we denote by $\phi_*: H_n(\G)
\to H_n(\HH)$, and $\phi \mapsto \phi_*$ preserves composition. If $\G$ is an open
subgroupoid of an ample Hausdorff groupoid $\HH$, then $\G$ is also an ample Hausdorff
groupoid and the inclusion map $\iota: \G \to \HH$ is an \'etale groupoid homomorphism.
Hence $\iota$ induces a map $\iota_* : H_* ( \G ) \to H_* ( \HH )$ satisfying $\iota_* [
1_U ]_{H_n ( \G )} = [ 1_U ]_{H_n ( \HH )}$.
\end{remark}

One key point of the functoriality of homology described in the preceding remark is that
it leads to the following notion of continuity for homology of ample Hausdorff groupoids.

\begin{proposition}[cf. \cite{ortega} Lemma 1.5] \label{prop:union-homology}
Let $\G$ be an ample Hausdorff groupoid and let $\{ \G_i \}$ be an increasing sequence of
open subgroupoids of $\G$.  Then
\[
H_* (\G) \cong \varinjlim H_* (\G_i ).
\]
\end{proposition}
\begin{proof}
For each $i$ the inclusion map $\G_i \hookrightarrow \G$ induces a homomorphism $\iota_i
: H_*(\G_i) \to H_*(\G)$ by Remark~\ref{rmk:functoriality}. So the universal property of
the direct limit yields a homomorphism $\iota_\infty : \varinjlim H_* (\G_i ) \to H_*
(\G)$. This homomorphism is injective because if $\iota_\infty(a)$ is a boundary, say
$\iota_\infty(a) = \partial_n(f)$, then $f \in C_c(\G_i^{(n+1)})$ for large enough $i$,
and then $a = \partial_n(f)$ belongs to $B_n(\G_i)$. It is surjective because if $U$ is a
compact open subset of $\G^{(n)}$, then $U \subseteq \bigcup_i \G^{(n)}_i$, and since the
$\G_i^{(n)}$ are open and nested, it follows that $U$ is a compact open subset of
$\G_i^{(n)}$ for large $i$. Hence every generator of $H_*(\G)$ belongs to the image of
$\iota_\infty$.
\end{proof}

\begin{lemma}\label{lem:sim}
Let $X, Y$ be locally compact Hausdorff spaces, and let $\psi : Y \to X$ be a local homeomorphism.
Suppose that there exists a continuous open section $\varphi : X \to Y$ of $\psi$.
Then the groupoid maps $\rho: R(\psi) \to X$ given by $\rho(y_1, y_2) = \psi(y_1)$ and
$\sigma: X \to R(\psi)$ given by $\sigma(x) = (\varphi(x), \varphi(x))$ are both similarities.
Indeed, $\sigma \circ \rho$ is similar to $\id_{R(\psi)}$ and $\rho \circ \sigma$ is similar to
$\id_{X}$.
Moreover, the induced maps $\rho_*: H_*(R(\psi)) \to H_*(X)$ and $\sigma_*: H_*(X) \to H_*(R(\psi))$
are inverse to each other.
\end{lemma}
\begin{proof}
The first assertion follows from the second. To prove the second assertion, define
$\theta: R(\psi)^{(0)} \to R(\psi)$ by $\theta(y, y) = (\varphi\circ\psi(y), y)$. Then
\begin{align*}
\sigma \circ \rho(y_1, y_2)\theta(y_2, y_2) &=  (\varphi\circ\psi(y_1), \varphi\circ\psi(y_2))(\varphi\circ\psi(y_2), y_2)   =  (\varphi\circ\psi(y_1), y_2) \\
&= (\varphi\circ\psi(y_1), y_1)(y_1, y_2) = \theta(y_1, y_1) \id_{R(\psi)}(y_1, y_2).
\end{align*}
Hence, $\sigma \circ \rho$ is similar to $\id_{R(\psi)}$.  Since $\rho \circ \sigma =  \id_{X}$, it follows that $\rho \circ \sigma$ is similar to $\id_X$.
The last assertion now follows from \cite[Proposition 3.5]{Matui-2012}.
\end{proof}

\begin{definition}[{cf. \cite[Definition~2.2]{Matui-2012}}]\label{def:AF}
An ample groupoid $\G$ is said to be \textit{elementary} if it is isomorphic to the
groupoid $R(\psi)$  of Example~\ref{ex:rpsidef} for some local homeomorphism $\psi: Y \to
X$ between 0-dimensional spaces. An ample groupoid $\G$ is said to be \textit{AF} if it
can be expressed as a union of open elementary subgroupoids.
\end{definition}

The only point of difference between Definition~\ref{def:AF} and Matui's
\cite[Definition~2.2]{Matui-2012} is that we allow non-compact unit spaces.

For the following result, recall that if $\psi : Y \to X$ is a local homeomorphism then
the homomorphism $\psi_* : C_c(Y, \Z) \to C_c(X, \Z)$ is given in \eqref{eq:homology induced map}.
There is also an inclusion $\iota : C_0(Y) \hookrightarrow C^*(R(\psi))$ induced
by the homeomorphism $Y \cong R(\psi)^{(0)}$, and this induces a homomorphism $\iota_* :
C_c(Y, \Z) \to K_0(C^*(R(\psi))$.

\begin{theorem}\label{thm:R(psi) homology}
Let $X, Y$ be locally compact Hausdorff spaces, and let $\psi : Y \to X$ be a local
homeomorphism. Then $Y$ is an $R(\psi)$--$X$ equivalence with anchor maps $\id : Y \to
R(\psi)^{(0)}$ and $\psi : Y \to X$, right action of $X$ given by $y \cdot \psi(y) = y$,
and left action given by $(x,y) \cdot y = x$. Hence $H_*(R(\psi)) \cong H_*(X)$.

If $Y$ is $\sigma$-compact and totally disconnected, then the map $\psi$ admits a continuous open section and the map
$\rho: R(\psi) \to X$ given by $\rho(y_1, y_2) = \psi(y_1)$ is a similarity and thus induces the isomorphism
$H_0(R(\psi)) \cong H_0(X)  = C_c(X, \Z)$ determined by $[1_U] \mapsto \psi_*(1_U)$ for $U \subseteq Y$ compact and open.
We have $H_n(R(\psi)) = 0$ for $n \ge 1$.

The groupoid $C^*$-algebra $C^*(R(\psi))$ is an AF algebra, the map $\psi_*$ induces an
isomorphism $K_0(C^*(R(\psi))) \to C_c(X, \Z)$ such that the diagram
\[
\begin{tikzpicture}
    \node (a) at (0, 2) {$C_c(Y, \Z)$};
    \node (b) at (0, 0) {$K_0(C^*(R(\psi)))$};
    \node (c) at (4, 0) {$C_c(X, \Z)$};
    \draw[->] (a)--(c) node[pos=0.5, above] {$\psi_*$};
    \draw[->] (a)--(b) node[pos=0.5, left] {$\iota_*$};
    \draw[->] (b)--(c) node[pos=0.5, above] {$\cong$};
\end{tikzpicture}
\]
commutes, and we have $K_1(C^*(R(\psi))) = \{0\}$.
\end{theorem}
\begin{proof}
For the first statement, one just checks directly that the maps described satisfy the
axioms for an equivalence of groupoids. The isomorphism $H_*(R(\psi)) \cong H_*(X)$
follows from Lemma~\ref{lem:morita-invariance}.

Now suppose that $Y$ is $\sigma$-compact and totally disconnected. Then $X$ is also
$\sigma$-compact and totally disconnected. Choose a cover $Y = \bigcup^\infty_{i=1} U_i$
of $Y$ by countably many compact open sets. Since $\psi$ is a local homeomorphism and the
$U_i$ are compact, each $U_i$ is a finite union of compact open sets on which $\psi$ is
injective, so by relabelling we may assume that the $U_i$ have this property. For each
$i$, let $V_i := U_i \setminus \big(\bigcup^{i-1}_{j=1} \psi^{-1}(\psi(U_j))\big)$. Then
$X = \bigsqcup_i \psi(V_i)$, and the $V_i$ are compact open sets on which $\psi$
restricts to a homeomorphism $\psi_i : V_i \to \psi(V_i)$. So we can define a continuous
section $\varphi$ for $\psi$ by setting $\varphi|_{\psi(V_i)} = \psi_i^{-1} : \psi(V_i)
\to V_i$. Hence Lemma~\ref{lem:sim} yields a similarity $\rho: R(\psi) \to X$ such that
the restriction of $\rho_*$ to $C_c(Y, \Z)$ coincides with $\psi_*: C_c(Y, \Z) \to C_c(X,
\Z)$.

The equivalence $Y$ of groupoids determines a
Morita equivalence between $C^*(R(\psi))$ and $C_0(X)$ \cite[Theorem~2.8]{MRW}. Since
approximate finite dimensionality is preserved by Morita equivalence and $C_0(X)$ is AF, we see that
$C^*(R(\psi))$ is AF, and the Morita equivalence induces the desired isomorphisms in
$K$-theory. To see that the diagram commutes, suppose that $U \subseteq Y$ is compact
open and that $\psi|_U$ is injective. Then $\iota_*(1_U) = [1_U] \in K_0(C^*(R(\psi)))$,
and this is carried to $1_{\psi(U)}$ by the isomorphism $K_0(C^*(R(\psi))) \to C_c(X,
\Z)$ just described. This is precisely $\psi_*(1_U)$.
\end{proof}

\section{Matui's HK Conjecture}\label{sec:HK}

In \cite[Conjecture 2.6]{Matui-2016} Matui posed the HK conjecture for a certain class of
ample Hausdorff groupoids.   Recall that an  \'{e}tale groupoid $\G$ is said to be
effective if the interior of its isotropy coincides with its unit space $\G^{(0)}$ and
minimal if every orbit is dense.

\begin{mconj}
Let $\G$ be a locally compact Hausdorff  \'{e}tale groupoid such that $\G^{(0)}$ is a
Cantor set. Suppose that $\G$ is both effective and minimal. Then for $j = 0, 1$ we have
\begin{equation}\label{conj:HK}
K_j(C^*_r(\G)) \cong \bigoplus_{i = 0}^{\infty} H_{2i+j}(\G).
\end{equation}
\end{mconj}

We are interested in the extent to which the isomorphism~\eqref{conj:HK} holds amongst
groupoids that do not necessarily have non-compact unit space and are not necessarily
minimal or effective. To streamline our discussion, we make the following definition.

\begin{definition}\label{dfn:class M}
We define $\M$ to be the class of ample Hausdorff groupoids for which the
isomorphism~\eqref{conj:HK} holds.
\end{definition}

Matui proves in \cite[Theorem 4.14]{Matui-2012} that the groupoids associated to shifts
of finite type belong to $\M$ and in \cite[Theorems 4.10~and~4.11]{Matui-2012} that AF
groupoids with compact unit space belong to $\M$. He shows in
\cite[Proposition~2.7]{Matui-2016} that $\M$ is closed under Kakutani equivalence, and he
shows in \cite[Theorem 5.5]{Matui-2016} that $\M$ contains all finite cartesian products
of groupoids associated to shifts of finite type. He proves in
\cite[Section~3.1]{Matui-2012} that $\M$ contains the transformation groupoids of
topologically free and minimal actions of $\Z$ on the Cantor set.

In \cite{HazLi} Hazrat and Li verify \eqref{conj:HK} for $j=0$ in the setting of
groupoids of row-finite $1$-graphs with no sinks. We complete the analysis for $j=1$ in
Theorem~\ref{thm:DR k=1}. In \cite{ortega} Ortega shows that the Katsura-Exel-Pardo
groupoid $\G_{A,B}$ associated to square integer matrices with $A \ge 0$ belongs to $\M$.
Here we consider Deaconu-Renault groupoids and thereby
study higher dimensional aspects not present in other cases.

There are examples of ample Hausdorff groupoids $\G$ that belong to $\M$ but either do
not have compact unit spaces or which are not necessarily effective or minimal. For
example, if $X$ is any noncompact totally disconnected space, then the groupoid $X \times
\Z$ satisfies none of these conditions, but belongs to $\M$. It is also easy to show that
$\Z^n$ is in $\M$ for all $n$. Let $\F_n$ denote the free group on $n$ letters. Then by
\cite[10.8.1]{B} and \cite[II.4.2]{brown} we have
\[
K_0(C^*_r(\F_n)) \cong H_0(\F_n) \cong \Z , \quad
K_1(C^*_r(\F_n)) \cong H_1(\F_n) \cong \Z^n
\]
and $H_i(\F_n)  = 0$ for all $i > 1$.  Hence $\F_n$ lies in $\M$. The integer Heisenberg
group also belongs to $\M$---see \cite[Corollary 1]{kodaka} and \cite[Example
8.24]{knudson}. On the other hand, not every ample Hausdorff groupoid belongs to $\M$:
for example, $\M$ contains no nontrivial finite cyclic group.

Here we expand the class of groupoids known to belong to $\M$. We show that all AF
groupoids, all Deaconu--Renault groupoids associated to actions of $\N$ or $\N^2$ on
0-dimensional spaces, and path groupoids associated to many one-vertex $k$-graphs belong
to $\M$.

By Proposition~\ref{prop:trivial-groupoid} and Remark~\ref{rmk:trivial-case}, any
$0$-dimensional space $X$ regarded as a trivial groupoid belongs to $\M$. More generally,
the following corollary to Theorem~\ref{thm:R(psi) homology} shows that all AF groupoids
belong to $\M$.

\begin{corollary}\label{cor:lim homology}
Let $\G$ be a groupoid that can be expressed as a direct limit $\G = \varinjlim \G_n$ of
open subgroupoids each of which is isomorphic to $R(\psi_n)$ for some local homeomorphism
$\psi_n : \G^{(0)} \to X_n$. Suppose that $\G^{(0)}$ is totally disconnected. Then there
are maps $\varphi_n : X_n \to X_{n+1}$ such that $\varphi_n \circ \psi_n = \psi_{n+1}$
for all $n$. We have $H_n(\G) = 0$ for $n \ge 1$, and $H_0(\G) \cong \varinjlim(H_0(X_n),
(\varphi_n)_*)$.

There is an isomorphism $K_0(C^*(\G)) \cong H_0(\G)$ that carries
$[1_U]_0$ to $[1_U]$ for each compact open $U \subseteq \G^{(0)}$, and we have
$K_1(C^*(\G)) = \{0\}$. In particular, $\G$ belongs to $\M$.
\end{corollary}

\begin{proof}
Each $R(\psi_n)^{(0)}$ is totally disconnected because it is an open subspace of $\G^{(0)}$,
and so Theorem~\ref{thm:R(psi) homology} shows that $H_p(R(\psi_n)) = 0$ for $p \ge 1$ and
all $n$, and that $H_0(R(\psi_n)) \cong C_c(X_n, \Z) \cong K_0(C^*(R(\psi_n)))$ with both
isomorphisms induced by $(\psi_n)_*$. Since homology and $K$-theory are continuous with
respect to inductive limits, the result follows.
\end{proof}

\section{Deaconu--Renault groupoids}\label{sec:DR groupoids}

In this section we first show that the homology of an ample higher-rank Deaconu-Renault
groupoid $\G ( X , \sigma )$ is  given by $H_* ( \Z^k , H_0 ( \G ( X , \sigma ) \times_c
\Z^k ))$ using spectral sequence arguments of Matui. In Theorem~\ref{thm:main H*
computation} we describe a complex which allows us to compute $H_* ( \Z^k , H_0 ( \G ( X
, \sigma ) \times_c \Z^k )$ by adapting techniques from \cite{evans}. We then use this
description and Kasparov's $K$-theory spectral sequence to prove that $\G ( X , \sigma )$
belongs to $\M$ when the rank is either one or two (see Theorems~\ref{thm:DR k=1},
\ref{thm:DR k=2}). Furthermore we also give formulas for computing the $K$-theory of $C^*
( \G ( X , \sigma ))$ in these cases.

Recall from Section~\ref{s:definitions} that if $\sigma$ is an action of $\N^k$ on a
locally compact Hausdorff space $X$ by local homeomorphisms, then we write $\widetilde{X}
:= X \times \Z^k$, and there is an action $\tilde\sigma$ of $\N^k$ by local
homeomorphisms on $\widetilde{X}$ given by $\tilde\sigma^q(x, p) = (\sigma^q(x), p+q)$.
Equation~\eqref{eq:skew groupoid iso} then defines an isomorphism of the skew-product
groupoid $\G(X, \sigma) \times_c \Z^k$ corresponding to the cocycle $c(x, m, y) = m$ onto
the Deaconu--Renault groupoid $\G(\widetilde{X}, \tilde{\sigma})$.

Our first result shows that $\G ( X ,\sigma ) \times_c \Z^k$ is equivalent to $c^{-1} (0)$;
this in turn allows us to compute its homology.

\begin{lemma}\label{lem:skew gpd homology}
Let $X$ be a locally compact Hausdorff totally disconnected space, and let $\sigma$ be an
action of $\N^k$ on $X$ by local homeomorphisms. The set $X \times \{0\} \subseteq X
\times \Z^k$ is a clopen $\G(\widetilde{X}, \tilde{\sigma})$-full subspace of
$\G(\widetilde{X},\tilde\sigma)^{(0)}$. The map $(x, 0, y) \mapsto ((x, 0), 0, (y, 0))$
is an isomorphism of $c^{-1}(0) \subseteq \G(X, \sigma)$ onto $\G(\widetilde{X},
\tilde{\sigma})|_{X \times \{0\}}$, and $\G(X, \sigma) \times_c \Z^k$ is an AF groupoid.

There is an isomorphism of $H_*(c^{-1}(0))$ onto $H_*(\G(\widetilde{X}, \tilde{\sigma}))$
that carries $[1_U]$ to $[1_{U \times \{0\}}]$ for every compact open $U \subseteq X$.
\end{lemma}

\begin{proof}
It is clear that $X \times \{ 0 \}$ is a clopen subset of $X \times \Z^k$, the unit space
of $\G ( \widetilde{X}, \tilde\sigma )$. We claim that it is full.  Fix $(x,n) \in X
\times \Z^k$ and write $n = n_+ - n_-$ where $n_+ , n_- \in \N^k$. Then there exists $y
\in X$ such that $\sigma^{n_-} (y) = \sigma^{n_+} (x)$. Set $\gamma = ( (x, 0) , n_+ -
n_- , (y,  n))$. By construction we have  $\tilde{\sigma}^{n_+} (x, 0) =
\tilde{\sigma}^{n_-} (y, n)$, and so $\gamma \in \G ( \widetilde{X}, \tilde\sigma )$.
Furthermore $r(\gamma)=(x, 0)  \in X \times \{ 0 \}$ and $s(\gamma) = (y, n)$, proving
the claim.

The map $(x, 0, y) \mapsto ((x, 0), 0, (y, 0))$ is clearly an injective homomorphism. To
show that it is surjective, let $\gamma  = ((x,m),m-n,(y,n)) \in
\G(\widetilde{X},\tilde\sigma)|_{X \times \{0\}}$. Then $m=n=0$, so $\gamma$ is in the range of $c^{-1} (0)$.

Since $\G(X, \sigma) \times_c \Z^k \cong \G ( \widetilde{X} , \tilde\sigma )$ via the
isomorphism given in~\eqref{eq:skew groupoid iso}, and since $X \times \{ 0 \}$ is
$\G(\widetilde{X}, \tilde{\sigma})$-full, it follows from Theorem~\ref{thm:equivalences}
that $\G(X, \sigma) \times_c \Z^k$ is equivalent to $\G(\widetilde{X}, \tilde{\sigma})|_{X
\times \{0\}}$. By the preceding paragraph $\G(\widetilde{X}, \tilde{\sigma})|_{X \times
\{0\}} \cong c^{-1} (0)$ . Since $c^{-1} (0)$ can be written as an increasing union of
the elementary groupoids $R ( \sigma^n )$, it is AF (see Corollary~\ref{cor:lim
homology}), and so $\G(X, \sigma) \times_c \Z^k$ is an AF groupoid also.

The final statement follows from Lemma~\ref{lem:morita-invariance}.
\end{proof}

To compute the homology of $c^{-1}(0)$, we decompose it as the increasing union of the
subgroupoids $R(\sigma^n)$ as $n$ ranges over $\N^k$.

\begin{lemma}\label{lem:H0 of skewprod as limit}
Let $X$ be a totally disconnected locally compact Hausdorff space, and let $\sigma$ be an
action of $\N^k$ on $X$ by surjective local homeomorphisms. There is an isomorphism
$\varinjlim(C_c(X, \Z), \sigma^n_*) \to H_0(c^{-1}(0))$ that takes $\sigma^{0,
\infty}_*(1_U)$ to $[1_U]$ for every compact open $U \subseteq X$. We have
\[
H_q(\G(X, \sigma) \times_c\Z^k) \cong
    \begin{cases}
        \varinjlim_{n \in \N^k}(C_c(X, \Z), \sigma^n_*) &\text{ if $q = 0$}\\
        0 &\text{ otherwise.}
    \end{cases}
\]
The isomorphism $H_0(\G(X, \sigma) \times_c\Z^k) \cong \varinjlim_{n \in \N^k}(C_c(X,
\Z), \sigma^n_*)$ intertwines the action of $\Z^k$ on $H_0(\G(X, \sigma) \times_c\Z^k)$
given by $p \cdot ((x, m, y), n) = ((x, m, y), n+p)$ with the action of $\Z^k$ on
$\varinjlim_{n \in \N^k}(C_c(X, \Z), \sigma^n_*)$ induced by the $\sigma^n_*$.
\end{lemma}

\begin{proof}
If $U, V \subseteq X$ are compact open sets on which $\sigma^n$ is injective, and if we
have $\sigma^n(U) = \sigma^n(V)$, then we have $[1_U] = [1_V]$ in $H_0(R(\sigma^n))$.
Since every $W \subseteq X$ can be expressed as a finite disjoint union $W = \bigsqcup_j
\sigma^n(U_j)$ where each $U_j$ is compact open and each $\sigma^n|_{U_j}$ is injective,
it follows that there is a unique homomorphism $\varphi_n : C_c(X, \Z) \to
H_0(R(\sigma^n))$ such that $\varphi_n(\sigma^n_*(1_U)) = [1_U]$ for all compact open
$U$. This $\varphi_n$ is the map induced by the weak equivalence of groupoids $\varphi_n
: R(\sigma^n) \to X$  (see Remark~\ref{rmk:weakeq}), so Proposition~\ref{prp:nonample
equiv} and Lemma~\ref{lem:morita-invariance} imply that it is an isomorphism. For $m,n
\in \N^k$, let $\iota_{m,n}$ be the inclusion map $R(\sigma^m) \hookrightarrow
R(\sigma^{m+n})$, and let $(\iota_{m,n})_* : H_0(R(\sigma^m)) \to H_0(R(\sigma^{m+n}))$
be the induced map in homology. Then we have a commuting diagram
\[
\begin{tikzpicture}[xscale=0.85,yscale=0.7]
    \node (tl) at (0, 3) {$C_c(X, \Z)$};
    \node (tr) at (4, 3) {$C_c(X, \Z)$};
    \node (bl) at (0, 0) {$H_0(R(\sigma^m))$};
    \node (br) at (4, 0) {$H_0(R(\sigma^{m+n}))$};
    \draw[->] (tl)--(tr) node[above, pos=0.5]{$\sigma^n_*$};
    \draw[->] (tl)--(bl) node[left, pos=0.5]{$\varphi_m$};
    \draw[->] (bl)--(br) node[below, pos=0.5]{$(\iota_{m,n})_*$};
    \draw[->] (tr)--(br) node[right, pos=0.5]{$\varphi_{m+n}$};
\end{tikzpicture}
\]
Since $c^{-1}(0)$ is the increasing union of the closed subgroupoids $R(\sigma^n)$,
continuity of homology gives $H_0(c^{-1}(0)) \cong \varinjlim(H_0(R(\sigma^m)),
\iota_*)$. So the universal properties of the direct limits prove the first statement.

The isomorphism~\eqref{eq:skew groupoid iso} of $\G(X, \sigma) \times_c \Z^k$ with
$\G(\widetilde{X}, \widetilde{\sigma})$ and Lemma~\ref{lem:skew gpd homology} show that
$H_*(\G(X, \sigma) \times_c \Z^k) \cong H_*(c^{-1}(0))$. So the preceding paragraph
proves that $H_0(\G(X, \sigma) \times_c\Z^k) \cong \varinjlim_{n \in \N^k}(C_c(X, \Z),
\sigma^n_*)$. We saw in the preceding paragraph that $c^{-1}(0)$ is an AF groupoid, so
Corollary~\ref{cor:lim homology} proves that $H_q(\G(X, \sigma) \times_c\Z^k) = 0$ for $q
\ge 1$. The final statement follows from direct computation of the maps involved.
\end{proof}

Matui's spectral sequence \cite[Theorem 3.8(1)]{Matui-2012} relates $H_*(\G(X, \sigma))$
to the homology of $\Z^k$ with coefficients in $H_*(\G(X, \sigma) \times_c \Z^k)$.
Since $H_q ( \G(X, \sigma) \times_c \Z^k)) =0$ for $q \ge 1$ (it is AF by Lemma~\ref{lem:H0 of skewprod as limit}),
the spectral sequence collapses: It follows that
\[
H_q ( \G ( X , \sigma ) ) \cong H_q (  \Z^k , H_0 ( \G (X,\sigma) \times_c \Z^k ) ) \text{ for } 0 \le q \le k .
\]
The proof of the following lemma is based on the technique developed in
\cite[Section~3]{evans}.  Recall that for $1 \le p \le k$, the module $\bigwedge^p R^k$
is the free $R$-module generated by the elements $\varepsilon_{i_1} \wedge \cdots \wedge
\varepsilon_{i_p}$ indexed by integer tuples $1 \le i_1 < i_2 < \cdots < i_p \le k$. We
define $\bigwedge^0 R^k := R$.

\begin{lemma}\label{lem:homology reduction}
Let $A$ be an abelian group, and suppose that $\sigma_1, \dots, \sigma_k$ are
pairwise commuting endomorphisms of $A$ (i.e. an action of $\N^k$). For
$1 \le p \le k$, define
\[\textstyle
\partial_p : \bigwedge^p \Z^k \otimes A \to \bigwedge^{p-1} \Z^k \otimes A
\]
on spanning elements $(\varepsilon_{i_1} \wedge \cdots \wedge \varepsilon_{i_p}) \otimes
a$ in which the $i_j$ are in strictly increasing order by
\[
\partial_p (\varepsilon_{i_1} \wedge \cdots \wedge \varepsilon_{i_p} \otimes a)
    = \begin{cases} \sum_j (-1)^{j+1} \varepsilon_{i_1} \wedge \cdots \wedge \widehat{\varepsilon}_{i_j} \wedge \cdots \wedge \varepsilon_{i_p} \otimes (\id - \sigma_{i_j}) a & \text{ for } p > 1 , \\
    1 \otimes (\id - \sigma_{i_1} ) a  & \text{ for } p= 1 .
    \end{cases}
\]
Then $\partial_{p-1} \circ \partial_p = 0$ for each $p$, so that $(\bigwedge^* \Z^k
\otimes A, \partial_*)$ is a complex. For each $0 \le i \le k$, the homomorphism $\id \otimes
\sigma_i : \bigwedge^* \Z^k \otimes A \to \bigwedge^* \Z^k \otimes A$ commutes with
$\partial_*$, and the induced map $(\id  \otimes \sigma_i)_*$ in homology is the identity
map.
\end{lemma}

\begin{proof}
Direct computation on a spanning element $( \varepsilon_{i_1} \wedge \cdots \wedge
\varepsilon_{i_{p+1}} ) \otimes a$ of $\bigwedge^* \Z^k \otimes A$ using that the
$\sigma_i$ commute shows that $\partial_{p+1} \circ \partial_p = 0$. Clearly  the
$\id  \otimes \sigma_i$ commute with $\partial_*$.

For the final statement, we first claim that for $x \in \bigwedge^p\Z^k \otimes A$ and $1
\le l \le k$, we have
\[
\partial_{p+1}(\varepsilon_l \wedge x) = -\varepsilon_l \wedge \partial_p(x) + (\id \otimes (\id - \sigma_l))(x).
\]
To prove this, it suffices to consider $x = ( \varepsilon_{i_1} \wedge \cdots \wedge
\varepsilon_{i_p} )  \otimes a$ where $1 \le i_1 < \cdots < i_p \le k$. So fix such an $x$
and fix $l \le k$. We consider two cases.

Case~1: $l \not= i_h$ for all $h$. Then there exists $0 \le j \le p$ such that $i_h < l$
for $h \le j$ and $i_h > l$ for $h
> j$. Then, using at both the first and the third steps that $\varepsilon_l \wedge
\varepsilon_{i_1} \wedge \cdots \wedge \varepsilon_{i_n} = (-1)^n \varepsilon_{i_1}
\wedge \cdots \wedge \varepsilon_{i_n} \wedge \varepsilon_l$ for any $i_1 < i_2 < \cdots
< i_n$, we calculate:
\begin{align*}
\partial_{p+1}&(\varepsilon_l \wedge x)\\
    &= (-1)^j \partial_{p+1}(\varepsilon_{i_1} \wedge \cdots \wedge \varepsilon_{i_j} \wedge \varepsilon_l \wedge \varepsilon_{i_{j+1}} \wedge \cdots \wedge \varepsilon_p \otimes a)\\
    &= (-1)^j\Big[\sum_{h=1}^j (-1)^{h+1} \varepsilon_{i_1} \wedge \cdots \wedge \widehat{\varepsilon}_{i_h} \wedge \cdots \wedge \varepsilon_{i_j} \wedge \varepsilon_l \wedge \varepsilon_{i_{j+1}} \wedge \cdots \wedge \varepsilon_p \otimes (\id - \sigma_{i_h})a\\
        &\qquad + (-1)^{j+2}\varepsilon_{i_1} \wedge \cdots \wedge \varepsilon_{i_j} \wedge \varepsilon_{i_{j+1}} \wedge \cdots \wedge \varepsilon_p \otimes (\id - \sigma_l)a\\
        &\qquad + \sum_{h=j+1}^p (-1)^{h+2} \varepsilon_{i_1} \wedge \cdots \wedge \varepsilon_{i_j} \wedge \varepsilon_l \wedge \varepsilon_{i_{j+1}} \wedge \cdots \wedge \widehat{\varepsilon}_{i_h} \wedge \cdots \wedge \varepsilon_p \otimes (\id - \sigma_{i_h})a\Big]\\
    &= \sum_{h=1}^p(-1)^h \varepsilon_l \wedge \varepsilon_{i_1} \wedge \cdots \wedge \widehat{\varepsilon}_{i_h} \wedge \cdots \wedge \varepsilon_p \otimes (\id - \sigma_{i_h})a\\
        &\qquad + \varepsilon_{i_1} \wedge \cdots \wedge \varepsilon_p \otimes (\id - \sigma_l)a\\
    &= \varepsilon_l \wedge \left( \sum_{h=1}^p (-1)^h \varepsilon_{i_1} \wedge \cdots \wedge \widehat{\varepsilon}_{i_h} \wedge \cdots \wedge \varepsilon_p  \otimes (\id - \sigma_{i_h})a\right)\\
        &\qquad + (\id \otimes (\id - \sigma_l))(\varepsilon_{i_1} \wedge \cdots \wedge \varepsilon_{i_p} \otimes a)\\
    &= -\varepsilon_l \wedge \partial_p(x) + (\id \otimes (\id - \sigma_l))(x).
\end{align*}

Case~2: $l = i_h$ for some $h$. Then $\partial_{p+1}(\varepsilon_l \wedge x) =
\partial_{p+1}(0) = 0$. So we must show that $\varepsilon_l \wedge \partial_p(x) = (\id
\otimes (\id - \sigma_l))(x)$. We have
\[
\varepsilon_l \wedge \partial_p(x)
    = \sum^p_{j=1} \varepsilon_l \wedge (-1)^{j+1} \varepsilon_{i_1} \wedge \cdots \wedge \widehat{\varepsilon}_{i_j} \wedge \cdots \wedge \varepsilon_p \otimes (\id - \sigma_{i_j})(a).
\]
The terms corresponding to $i_j \not= l$ are zero, so this collapses to
\begin{align*}
\varepsilon_l \wedge \partial_p(x)
    &= (-1)^{h+1} \varepsilon_l \wedge \varepsilon_{i_1} \wedge \cdots \wedge \widehat{\varepsilon}_{i_h} \wedge \cdots \wedge \varepsilon_{i_p} \otimes (\id - \sigma_{i_h})(a)\\
    &=  \varepsilon_{i_1} \wedge \cdots \wedge \varepsilon_{i_p} \otimes (\id - \sigma_l)(a)\\
    &= (\id \otimes (\id - \sigma_l))(x).
\end{align*}
This completes the proof of the claim.

We now prove the final statement. Fix $x \in \bigwedge^p\Z^k \otimes A$ such that
$\partial_p(x) = 0$, and fix $l \le p$. We just have to show that $(\id \otimes (\id -
\sigma_l))(x) \in \operatorname{image}(\partial_{p+1})$. Since $\partial_p(x) = 0$ and
using the claim, we see that
\[
(\id \otimes (\id - \sigma_l))(x)
    = (\id \otimes (\id - \sigma_l))(x) - \varepsilon_l \wedge \partial_p(x)
    = \partial_{p+1}(\varepsilon_l \wedge x),
\]
and the result follows.
\end{proof}

\begin{lemma}\label{lem:homology with limits}
Let $A$ be an abelian group, and suppose that $\sigma_1, \dots, \sigma_k$ are pairwise
commuting endomorphisms of $A$ (i.e. an action of $\N^k$). Let $\partial_p : \bigwedge^*
\Z^k \otimes A \to \bigwedge^* \Z^k \otimes A$ be as in Lemma~\ref{lem:homology
reduction}. Let $\widetilde{A} := \varinjlim_{\N^k}(A, \sigma^n)$. For $i \le k$ let
$\tilde\sigma_i$ be the automorphism of $\widetilde{A}$ induced by $\sigma_i$, and let
$\tilde\partial_p : \bigwedge^* \Z^k \otimes \widetilde{A} \to \bigwedge^* \Z^k \otimes
\widetilde{A}$ be the boundary map obtained from Lemma~\ref{lem:homology reduction}
applied to $\widetilde{A}$ and the $\tilde\sigma_i$. Then the canonical homomorphism
$\sigma^{0,\infty} : A \to \widetilde{A}$ corresponding to the $0$\textsuperscript{th}
copy of $A$ induces an isomorphism $H_*(\bigwedge^*\Z^k \otimes A) \cong
H_*(\bigwedge^*\Z^k \otimes \widetilde{A})$. Moreover, $\tilde\sigma$ extends to an
action of $\Z^k$ on $\widetilde{A}$, and $H_*(\bigwedge^*\Z^k \otimes \widetilde{A})
\cong H_*(\Z^k, \widetilde{A})$.
\end{lemma}

\begin{proof}
Since the homology functor is continuous (see, for example,
\cite[Theorem~4.1.7]{spanier}) we have $H_*(\bigwedge^*\Z^k \otimes \widetilde{A}) \cong
\varinjlim\big(H_*(\bigwedge^*\Z^k \otimes A), (\id \otimes \sigma^n)_*\big)$.
Lemma~\ref{lem:homology reduction} shows that this is equal to
$\varinjlim\big(H_*(\bigwedge^*\Z^k \otimes A), \id\big) = H_*(\bigwedge^*\Z^k \otimes
A)$.

For the second statement, we follow the argument of \cite[Lemma~3.12]{evans} (see also
\cite[Theorem~5.5]{KPS1}). Let $G := \Z^k = \langle s_1, \dots, s_k\rangle$ and let $R :=
\Z G$. For $p \ge 2$, we define $\partial_p : \bigwedge^p R^k \to \bigwedge^{p-1} R^k$ by
\[
\partial_p (\varepsilon_{i_1} \wedge \cdots \wedge \varepsilon_{i_p})
    = \sum_j (-1)^{j+1} (1 - s_{i_j}) \varepsilon_{i_1} \wedge \cdots \wedge
    \widehat{\varepsilon}_{i_j} \wedge \cdots \wedge \varepsilon_{i_p}.
\]
Define $\partial_1 : \bigwedge^1 R^k \to R$ by $\partial_1(\varepsilon_j) := 1 - s_j$,
and let $\eta : R \to \Z$ be the augmentation homomorphism determined by $\eta(s_i) = 1$
for each $i$. Then
\[\textstyle
0 \longrightarrow \bigwedge^k R^k \xrightarrow{\,\partial_k\,} \cdots
    \xrightarrow{\,\partial_2\,} \bigwedge^1 R^k \xrightarrow{\,\partial_1\,} R
    \xrightarrow{\,\eta\,} \Z
\]
is a free resolution of $\Z$. Hence, by definition of homology with coefficients in the
$\Z^k$-module $\widetilde{A}$ \cite[Equation~III(1.1)]{brown}, we have $H_*(\Z^k,
\widetilde{A}) \cong H_*(\bigwedge^* R^k \otimes_R \widetilde{A})$. As a group we have
$\bigwedge^p R^k \otimes_R \widetilde{A} \cong \bigwedge^* \Z^k \otimes \widetilde{A}$,
and this isomorphism intertwines the boundary maps in the complex defining
$H_*(\bigwedge^* R^k \otimes_R \widetilde{A})$ with the maps $\widetilde{\partial}_p$.
\end{proof}

We can now state our main theorem for this section, which is a computation of the
homology of the Deaconu--Renault groupoid $\G(X, \sigma)$ associated to an action of
$\N^k$ by surjective local homeomorphisms of a totally disconnected locally compact space
$X$.

\begin{theorem}\label{thm:main H* computation}
Let $X$ be a second-countable totally disconnected locally compact space, and let
$\sigma$ be an action of $\N^k$ by surjective local homeomorphisms $\sigma_p : X \to X$.
For $1 \le p \le k$, let
\[\textstyle
A^\sigma_p = \bigwedge^p \Z^k \otimes C_c(X, \Z)
\]
and let $A^\sigma_p = \{ 0 \}$ for $p > k$. For $p \ge 1$, define $\partial_p :
A^\sigma_p \to A^\sigma_{p-1}$ by
\[
\partial_p(\epsilon_{i_1} \wedge \dots \wedge \epsilon_{i_p} \otimes f)
    = \begin{cases}
    \id \otimes (\id - \sigma_*^{e_{i_1}}) f  & \text{if $p= 1$,}\\
    \displaystyle \sum^p_{j=1} (-1)^{j+1} \epsilon_{i_1} \wedge \dots \wedge \widehat{\epsilon}_{i_j} \wedge \dots \wedge \epsilon_{i_p} \otimes (\id - \sigma^{e_{i_j}}_*)f & \text{if $2 \le p \le k$,} \\
    0 &\text{if $p \ge k+1$.}
    \end{cases}
\]
Then $(A^\sigma_*, \partial_*)$ is a complex, and
\[
H_*(\G(X, \sigma)) \cong H_*(\Z^k, H_0(\G(X, \sigma) \times_c \Z^k)) \cong H_*(A^\sigma_*, \partial_*).
\]
We have $H_p ( (\Z^k, H_0(\G(X, \sigma) \times_c \Z^k)) = 0$ for $p > k$.
\end{theorem}

\begin{proof}
Lemma~\ref{lem:homology reduction} implies that $(A^\sigma_*, \partial_*)$ is a complex.

The automorphisms $\alpha_p : ((x, m, y), n) \mapsto ((x, m, y), n+p)$ of $\G(X, \sigma)
\times_c \Z^k$ induce an action $\widetilde{\alpha}_p$ of $\Z^k$ on $H_*(\G(X, \sigma)
\times_c \Z^k)$. This action $\widetilde{\alpha}$ makes each $H_p(\G(X, \sigma) \times_c
\Z^k)$ into a $\Z^k$-module, so it makes sense to discuss the homology groups $H_*(\Z^k,
H_q(\G(X, \sigma) \times_c \Z^k))$ of $\Z^k$ with coefficients in these modules. By
\cite[Theorem~3.8(1)]{Matui-2012} applied to the cocycle $c : \G(X, \sigma) \to \Z^k$,
there is a spectral sequence $E^r_{p,q}$ converging to $H_{p+q}(\G(X, \sigma))$
satisfying $E^2_{p,q} = H_p(\Z^k, H_q(\G(X, \sigma) \times_c \Z^k))$. Lemma~\ref{lem:H0
of skewprod as limit} shows that $H_q(\G(X, \sigma) \times_c \Z^k)$ is zero for $q \not=
0$, and therefore the differential maps on the $E^2$ page and above of the spectral
sequence are trivial. Hence $E^\infty_{p,q} = E^2 _{p,q}$ for all $p,q$ and so
\cite[Theorem 3.8(1)]{Matui-2012} shows that $H_p( \G ( X , \sigma ) ) \cong E^2_{p,0}
\cong H_p ( \Z^k , H_0 ( \G ( X , \sigma ) \times_c \Z^k ) )$ for all $p$.

The last statement of Lemma~\ref{lem:H0 of skewprod as limit} gives $H_*(\Z^k, H_0(\G(X,
\sigma) \times_c \Z^k)) \cong H_*(\Z^k, \varinjlim(X, \sigma^n_*))$.
Lemma~\ref{lem:homology with limits} shows that this is isomorphic to $H_*(A^\sigma_*,
\partial_*)$.
\end{proof}

Our next two results show that $\G(X, \sigma)$  belongs to $\M$ if $k \le 2$. For $k = 1$
we use the Pimsner--Voiculescu sequence \cite[Theorem 2.4]{pv}. To prove them, we need
the following lemma, which follows from a standard argument.

\begin{lemma}\label{lem:Takai duality}
Let $X$ be a second-countable locally compact totally disconnected space. Let $\sigma$ be
an action of $\N^k$ on $X$ by surjective local homeomorphisms. Let $c : \G(X, \sigma) \to
\Z^k$ be the cocycle $c(x, m, y) = m$. Then the action of $\Z^k$ on $\G(X, \sigma)
\times_c \Z$ given by $\alpha_p((x, m, y), n) = ((x, m, y), n+p)$ induces an action
$\bar{\alpha}$ of $\Z^k$ on $C^*(\G(X, \sigma) \times_c \Z)$. The crossed product
$C^*(\G(X, \sigma) \times_c \Z) \rtimes_{\bar{\alpha}} \Z^k$ is stably isomorphic to
$C^*(G(X, \sigma))$.
\end{lemma}
\begin{proof}
Every automorphism of a groupoid induces an automorphism of its $C^*$-algebra, and then
simpler calculations establish the first statement. Let $\gamma : \T^k \to \Aut(C^*(\G(X,
\sigma)))$ be the gauge action $\gamma_z(f(x, m, y)) = z^m f(x, m, y)$. Then there is an
isomorphism $\theta : C^*(\G(X, \sigma) \times_c \Z^k) \to C^*(G(X, \sigma))
\times_\gamma \T^k$ such that for $f \in C_c(\G(X, \sigma) \times \{m\}) \subseteq
C_c(\G(X, \sigma) \times_c \Z^k)$, we have
\[
\big(\theta(f)(z)\big)(g) = z^m f(g, m).
\]
This $\theta$ intertwines $\bar{\alpha}$ and the dual action $\widehat\gamma$ on
$C^*(G(X, \sigma)) \times_\gamma \T^k$, and so Takesaki--Takai duality implies that
$C^*(\G(X, \sigma) \times_c \Z^k) \rtimes_{\bar{\alpha}} \Z^k$ is stably isomorphic to
$C^*(\G(X, \sigma))$ (see \cite[Theorem 4.5]{Takesaki}, \cite[Theorem 3.4]{Takai}, and
\cite[Theorem~1.2]{bgr}).
\end{proof}

\begin{theorem}\label{thm:DR k=1}
Let $X$ be a second-countable locally compact totally disconnected locally compact space,
let $\sigma : X \to X$ be a surjective local homeomorphism, and let $\sigma_* : C_c ( X ,
\Z ) \to C_c ( X , \Z )$ be the induced map. Then
\begin{align*}
K_0 ( C^* ( \G ( X , \sigma ))   &\cong   H_0 ( \G (X, \sigma ) )  \cong \coker (\id - \sigma_* ) ,\\
K_1 ( C^* ( \G ( X , \sigma ) ) &\cong H_1 ( \G ( X, \sigma ) )  \cong \ker (\id - \sigma_* ) , \\
\text{and } H_n ( \G ( X, \sigma ) ) &= 0 \text{ for } n \ge 2.
\end{align*}
In particular, $\G(X, \sigma)$ belongs to $\M$.
\end{theorem}
\begin{proof}
We first calculate the $K$-theory of $C^* ( \G(X, \sigma) ) $. Lemma~\ref{lem:Takai
duality} applied with $k = 1$ shows that $K_*(C^*(\G(X, \sigma))) \cong K_*(C^*(\G(X,
\sigma) \times_c \Z) \rtimes_{\bar{\alpha}} \Z)$. Corollary~\ref{cor:lim homology} shows
that $K_1(C^*(\G(X, \sigma) \times_c \Z)) = 0$, and so exactness of the
Pimsner--Voiculescu sequence \cite[Theorem 2.4]{pv} implies that
\begin{equation} \label{eq:kokoko}
K_0(C^*(\G(X, \sigma))) \cong \coker(\id - \bar{\alpha}_*)
    \quad\text{ and }\quad
K_1(C^*(\G(X, \sigma))) \cong \ker(\id - \bar{\alpha}_*).
\end{equation}

We now compute the homology of $\G(X, \sigma)$. Theorem~\ref{thm:main H* computation}
shows that $H_p(\G(X, \sigma)) \cong H_p(\Z, H_0(\G(X, \sigma) \times_c \Z))$ for all
$p$. Moreover $H_p ( \Z, H_0(\G(X, \sigma) \times_c \Z)) = 0$ for $p \ge 2$ by
Theorem~\ref{thm:main H* computation}.

Let $\alpha_*$ be the action of $\Z$ on $H_0 ( \G(X, \sigma) \times_c \Z )$ induced by
$\alpha$. By \cite[Example III.1.1]{brown},
\begin{equation}\label{eq:hohoho}
\begin{split}
 H_0(\Z, H_0(\G(X, \sigma) \times_c \Z)) &\cong \ker (\id - \alpha_*) , \quad \text{ and } \\
 H_1(\Z, H_0(\G(X, \sigma) \times_c \Z)) &\cong \coker (\id - \alpha_*).
\end{split}
\end{equation}

Recall from Lemma~\ref{lem:skew gpd homology} that $\G(X, \sigma) \times_c \Z$ is an AF groupoid.
The isomorphism between $H_0(\G(X, \sigma)
\times_c \Z)$ and $K_0(C^*(\G(X, \sigma) \times_c \Z))$ supplied by
Corollary \ref{cor:lim homology} intertwines the $\Z$-actions $\alpha_*$ and
 $\bar\alpha_*$. Thus $\coker(\id - \bar{\alpha}_*) \cong \coker(\id - \alpha_*)$ and
 $\ker(\id - \bar{\alpha}_*) \cong \ker(\id - \alpha_*)$. Hence by \eqref{eq:kokoko} and \eqref{eq:hohoho},
\[
K_0 ( C^* ( \G ( X , \sigma ))   \cong   H_0 ( \G (X, \sigma ) )  \quad \text{and} \quad
K_1 ( C^* ( \G ( X , \sigma ) ) \cong H_1 ( \G ( X, \sigma ) ) .
\]
Since $k=1$ here, the complex $A_*^\sigma$ of Theorem~\ref{thm:main H* computation}
reduces to
\[
0 \longrightarrow C_c ( X , \Z ) \longrightarrow C_c ( X , \Z ) \longrightarrow 0
\]
where the central map is $\id - \sigma_*$. Then we have
\begin{align*}
H_0 ( \G ( X , \sigma ) ) &\cong H_0 ( A^\sigma_* , \partial_* ) \cong \coker (\id - \sigma_*) , \text{ and } \\
H_1 ( \G ( X , \sigma ) ) &\cong H_1 ( A^\sigma_* , \partial_* ) \cong \ker (\id - \sigma_*)
\end{align*}
as required.
\end{proof}

\begin{remark}\label{rmk:1graphHK}
Matui's  paper \cite{Matui-2012}, together with recent results by Hazrat and Li
\cite{HazLi}  and Ortega \cite{ortega} suggest that the groupoids of (not necessarily
finite) $1$-graphs with no sources belong to $\M$. This now follows from
Theorem~\ref{thm:DR k=1}, since graph groupoids are, by definition, rank-1
Deaconu--Renault groupoids.
\end{remark}

We now discuss  Kasparov's $K$-theory spectral sequence for $C^* ( \G ( X , \sigma ) )$.
Lemma~\ref{lem:Takai duality} shows that $C^* ( \G ( X , \sigma ) )$ is stably isomorphic
to the crossed product $C^*(\G(X, \sigma) \times_c \Z^k) \times_{\bar{\alpha}} \Z^k$.
Hence Theorem~6.10 of \cite{kasparov} (see also \cite[\S 3]{evans}) shows that there is a
spectral sequence $E^r_{p,q}$ converging to $K_*(C^* ( \G ( X , \sigma ) ))$ with
$E^2$-page given by
\[
E^2_{p,q} = \begin{cases} H_p ( \Z^k , K_q ( C^* ( \G (X, \sigma ) ) \times_c \Z^k ) )
& \text{ if } q \text{ is even, and } 0 \le p \le k \\
0 & \text{ otherwise.}
\end{cases}
\]
The differential maps in the spectral sequence are maps $d^r_{p,q} : E^r_{p,q} \to
E^r_{p-r,q+r-1}$. If $r > k$, then for any $p,q$ at least one of $E^r_{p,q}$ and
$E^r_{p-r, q+r-1}$ is trivial,  because $E^r_{p,q}$ is nontrivial only for $0 \le p \le
k$. Hence $d^r_{p,q}$ is trivial for $r > k$. Thus $E^\infty_{p,q} = E^{k+1}_{p,q}$. If
$k$ is even, we can improve on this: if $r = k$ is even, then at least one of $E^r_{p,q}$
and $E^r_{p-r, q+r-1}$ is trivial because $E^r_{p,q}$ is nontrivial only for $q$ even,
and it follows that $E^\infty_{p,q} = E^k_{p,q}$ for all $p,q$. In particular, if $k =
2$, then we have $E^\infty_{p,q} = E^2_{p,q}$.

For our next theorem, we need the well-known fact that if $X$ is locally compact
Hausdorff space, then $C_c(X, \Z)$ is a free abelian group. We provide a proof for
completeness.

\begin{lemma}\label{lem:Cc(X,Z) freel abelian}
Let $X$ be a second-countable locally compact Hausdorff space. Then $C_c(X, \Z)$ is a
free abelian group.
\end{lemma}
\begin{proof}
First note that if $X$ is not compact, then $C(X \cup \{\infty\}, \Z) \cong C_c(X,\Z)
\oplus \Z$ via $f \mapsto (f - f(\infty)1, f(\infty))$, so it suffices to prove the
result for $X$ compact.

So suppose that $X$ is compact. Since $X$ is metrisable by the Urysohn metrisation
theorem (see \cite[Theorems 23.1~and~17.6(a)]{Willard}), the Alexandroff--Hausdorff
theorem (see \cite[Theorem~30.7]{Willard}) shows that there is a continuous surjection
$\phi : \{0,1\}^\infty \to X$. Hence $\phi^* : C(X, \Z) \to C(\{0,1\}^\infty, \Z)$ is an
injective group homomorphism. Since subgroups of free abelian groups are themselves free
abelian, it therefore suffices to show that $C(\{0,1\}^\infty, \Z)$ is free abelian.

For this, let $\{0,1\}^*$ denote the collection of all finite words in the symbols $0,1$,
including the empty word $\varepsilon$. Let $I = \{\varepsilon\} \cup \{\omega 1 : \omega
\in \{0,1\}^*\}$ denote the subset of $\{0,1\}^*$ consisting of the empty word and all
nontrivial words that end with a 1. We claim that $B := \{1_{Z(\omega)} : \omega \in I\}$
is a family of free abelian generators of $C(\{0,1\}^\infty, \Z)$. To see this, we first
argue by induction on $n$ that  $\lsp_\Z\{1_{Z(\omega)} : \omega \in I\text{ and
}|\omega| \le n\} = \lsp_\Z\{1_{Z(\omega)} : \omega \in \{0,1\}^n\}$ for all $n$. The
containment $\subseteq$ is trivial. The containment $\supseteq$ is also trivial for $n =
0$, and if it holds for $n = k$, then for each $\omega = \omega'0 \in \{0,1\}^{k+1}$ that
ends in a 0, we have $1_{Z(\omega)} = 1_{Z(\omega')} - 1_{Z(\omega' 1)}$. We have
$\omega'1 \in I$ with $|\omega'1| = k+1$, and $1_{Z(\omega')} \in \lsp_\Z\{1_{Z(\omega)}
: \omega \in I\text{ and }|\omega| \le k\}$ by the inductive hypothesis. So the
containment $\supseteq$ also holds for $n = k+1$. Hence $B$ generates $C(\{0,1\}^\infty,
\Z)$ as a group. To see that $B$ is a family of free generators, suppose for
contradiction that $F \subseteq I$ is a finite set and $\{a_\omega : \omega \in F\}$ are
nonzero integers such that $\sum_{\omega \in F} a_\omega 1_{V_\omega} = 0$. Fix $\mu \in
F$ of minimal length. Then $\mu 000\cdots \in \{0,1\}^\infty$ belongs to $Z(\mu)$ but not
to $Z(\omega)$ for any $\omega \in F \setminus \{\mu\}$. Hence $0 = \big(\sum_{\omega \in
F} a_\omega 1_{V_\omega}\big)(\mu 000\cdots) = a_\mu$ contradicting the assumption that
the $a_\omega$ are nonzero.
\end{proof}

\begin{theorem}\label{thm:DR k=2}
Let $X$ be a second-countable locally compact totally disconnected space. Let $\sigma$ be
an action of $\N^2$ on $X$ by surjective local homeomorphisms. Define $d_2 : C_c(X, \Z)
\to C_c(X, \Z) \oplus C_c(X, \Z)$ by $d_2(f) = ((\sigma_*^{e_2} - \id)f, (\id -
\sigma_*^{e_1})f)$ and define $d_1 : C_c(X, \Z) \oplus C_c(X, \Z) \to C_c(X, \Z)$ by
$d_1(f \oplus g) = (\id - \sigma_*^{e_1})f + (\id - \sigma^{e_2}_*)g$. Then
\begin{align*}
K_0(C^*( \G(X, \sigma)) &\cong H_0( \G(X, \sigma)) \oplus H_2( \G(X, \sigma)) \cong \coker(d_1) \oplus \ker(d_2),\quad\text{ and}\\
K_1(C^*( \G(X, \sigma)) &\cong H_1(\ G(X, \sigma)) \cong \ker(d_1)/\operatorname{image}(d_2).
\end{align*}
In particular, $\G(X, \sigma)$ belongs to $\M$.
\end{theorem}
\begin{proof}
Let $A := K_0(C^*(\G(X, \sigma) \times_c \Z^2)) = H_0(\G(X, \sigma) \times_c \Z^2)$.
Lemma~\ref{lem:Takai duality} applied with $k = 2$ shows that $C^*(\G(X, \sigma))$ is
stably isomorphic to $C^*(\G(X, \sigma) \times_c \Z^2) \times_{\bar{\alpha}} \Z^2$ for
the action $\bar\alpha$ induced by translation in $\Z^2$ in the skew-product $\G(X,
\sigma) \times_c \Z^2$.

We follow the argument of \cite{evans, KPS1}. As discussed above, Kasparov's spectral
sequence \cite[Theorem 6.10]{kasparov} for $K_*(C^*(\G(X, \sigma) \times_c \Z^2)
\times_{\bar{\alpha}} \Z^2)$ converges on the second page, and we deduce that $K_0(C^*(\G(X,
\sigma))$ is an extension of $E^2_{0, 2}$ by $E^2_{2, 0}$ while $K_1(C^*(\G(X, \sigma))$
is isomorphic to $E^2_{0, 1}$.

As discussed prior to the statement of the theorem, $E^2_{p,q}$ is isomorphic to the
homology group $H_p(\Z^2, K_0(C^*(\G(X, \sigma)) \times_c \Z^k))$ for $q$ even and is
zero for $q$ odd. Corollary~\ref{cor:lim homology} shows that $K_0(C^*(\G(X, \sigma))
\times_c \Z^2) \cong H_0(\G(X, \sigma) \times_c \Z^2)$, and that this isomorphism
intertwines the actions of $\Z^2$ on the two groups induced by translation in the second
coordinate. It therefore follows from Theorem~\ref{thm:main H* computation} that for $q$
even, we have $E^2_{p,q} \cong H_p(A^\sigma_*, \partial_*)$. Since $C_c(X, \Z)$ is free
abelian by Lemma~\ref{lem:Cc(X,Z) freel abelian}, so is the subgroup $H_2(A^\sigma_*,
\partial_*) = \ker(\partial_2)$. Hence the extension $K_0(C^*(\G(X, \sigma))$ of $E^2_{0,
2}$ by $E^2_{2, 0}$ splits, and we obtain $K_0(C^*(\G(X, \sigma)) \cong H_0(A^\sigma_*,
\partial_*) \oplus H_2(A^\sigma_*, \partial_*)$ and $K_1(C^*(\G(X,\sigma))) \cong
H_1(A^\sigma_*, \partial_*)$. The result then follows from Theorem~\ref{thm:main H*
computation} because the obvious identifications $\bigwedge^j \Z^2 \otimes C(X, \Z) \cong
C(X, \Z)^{\binom{2}{j}}$ for $0 \le j \le 2$ intertwine $\partial_*$ with $d_*$.
\end{proof}

It follows that the path groupoid of a $2$-graph belongs to $\M$ since it is a rank-$2$
Deaconu-Renault groupoid (see Corollary~\ref{cor:HK for graphs}).

\begin{question}
Our proof Theorem~\ref{thm:DR k=2} uses that $H_2(A^\sigma_*)$ is a free abelian group so
that the extension $0 \to H_0(\G(X,\sigma)) \to K_0(C^*(\G(X, \sigma))) \to H_2(\G(X,
\sigma)) \to 0$ splits. Hence the map from $H_0(\G(X, \sigma)) \oplus H_2(\G(X, \sigma))$
to $K_0(C^*(\G(X, \sigma)))$ that we obtain is not natural. An interesting question
arises: can the isomorphism~\eqref{conj:HK} be chosen to be natural in some sense for
elements of $\M$ in general, and for rank-2 Deaconu--Renault groupoids in particular?
\end{question}

\begin{remark}
The proof of Theorem~\ref{thm:DR k=2} is special to the situation $k = 2$, and issues
arise already when $k = 3$. In this situation, the groups on the $E^3$-page of Kasparov's spectral
sequence coincide with those on the $E^2$-page, but the $E^3$-page has potentially nontrivial differential maps,
$d^3_{3,2l} : E^3_{3,2l} \to E^3_{0, 2l+2}$. So $E^3_{p,q} = H_p(A^\sigma_*)$ if $q$ is
even, and is 0 if $q$ is odd, and the $E^3$-page has the following form.
\[
\begin{tikzpicture}[yscale=0.4]
    \node (00) at (0, 0) {$H_0(A^\sigma_*)$};
    \node (10) at (3, 0) {$H_1(A^\sigma_*)$};
    \node (20) at (6, 0) {$H_2(A^\sigma_*)$};
    \node (30) at (9, 0) {$H_3(A^\sigma_*)$};
    \node (40) at (12,0) {$0$};
    \node (01) at (0, 3) {$0$};
    \node (11) at (3, 3) {$0$};
    \node (21) at (6, 3) {$0$};
    \node (32) at (9, 3) {$0$};
    \node (41) at (12,3) {$0$};
    \node (02) at (0, 6) {$H_0(A^\sigma_*)$};
    \node (12) at (3, 6) {$H_1(A^\sigma_*)$};
    \node (22) at (6, 6) {$H_2(A^\sigma_*)$};
    \node (32) at (9, 6) {$H_3(A^\sigma_*)$};
    \node (42) at (12,6) {$0$};
    \node (03) at (0, 9) {$0$};
    \node (13) at (3, 9) {$0$};
    \node (23) at (6, 9) {$0$};
    \node (33) at (9, 9) {$0$};
    \node (43) at (12,9) {$0$};
    \node at (14,0) {$\dots$};
    \node at (14,3) {$\dots$};
    \node at (14,6) {$\dots$};
    \node at (14,9) {$\dots$};
    \node at (0, 12) {$\vdots$};
    \node at (3, 12) {$\vdots$};
    \node at (6, 12) {$\vdots$};
    \node at (9, 12) {$\vdots$};
    \node at (12, 12) {$\vdots$};
    \draw[->] (30)--(02) node[pos=0.5, above] {$d^3_{3,0}$};
\end{tikzpicture}
\]
The sequence converges on the $E^4$-page, and so we have exact sequences
\begin{align*}
0 \to \coker(d^3_{3,0}) \to &K_0(C^*(\G(X, \sigma))) \to H_2(A^\sigma_*) \to 0
    \quad\text{ and }\\
0 \to H_1(A^\sigma_*) \to &K_1(C^*(\G(X, \sigma))) \to \ker(d^3_{3,0}) \to 0.
\end{align*}
Unless $d^3_{3,0}$ is trivial, there is no reason to expect that $\coker(d^3_{3,0}) \cong
H_0(A^\sigma_*)$; and even if $d^3_{3,0}$ is trivial, there is no reason to expect that
$H_2(A^\sigma_*)$ is free abelian, so the extension defining $K_0(C^*(\G(X, \sigma)))$
need not split. This suggests rank-3 Deaconu--Renault groupoids as a potential source of
counterexamples to Matui's HK-conjecture.
\end{remark}

\begin{remark}
If the groups $H_*(A^\sigma_*)$ are finitely generated, and the natural homomorphism
$H_0(\G(X, \sigma)) \to K_0(C^*(\G(X, \sigma)))$ is injective, then one would expect
$d^3_{3,0}$ to be trivial, and then in the rank-3 case, $\G(X, \sigma)$ would satisfy
Matui's conjecture up to stabilisation by $\Q$ (the so-called \emph{rational
HK-conjecture}). This suggests that it would be worthwhile to investigate when the
homomorphism $H_0(\G(X, \sigma)) \to K_0(C^*(\G(X, \sigma)))$ is injective.
\end{remark}

\section{\texorpdfstring{$k$}{k}-graphs}\label{sec:k-graphs}

In this section, we first establish the existence of a natural map from the homology of a
$k$-graph to the homology of its groupoid. We show  that this homomorphism is
in general neither injective nor surjective. We then apply the results of
Section~\ref{sec:DR groupoids} to see that all $1$-graph and $2$-graph groupoids belong
to $\M$. Finally, we restrict our attention to $k$-graphs with one vertex, and
demonstrate that for any such $k$-graph in which $\gcd(|\Lambda^{e_1}| - 1, \dots,
|\Lambda^{e_k}| -1) = 1$, the corresponding $k$-graph groupoid belongs to $\M$.

\subsection{A map from the categorical homology of a \texorpdfstring{$k$}{k}-graph to the homology of its groupoid}

To define the \emph{categorical}\footnote{so called because it matches with the
categorical cohomology of $\Lambda$ defined in \cite{KPS2}.} homology groups
$H_*(\Lambda)$ for a $k$-graph $\Lambda$, we use the following notation.

Given a $k$-graph $\Lambda$, let
\[
\Lambda^{*\,n} :=\begin{cases}
    \Lambda^0 &\text{ if $n = 0$}\\
    \{(\lambda_1,\ldots,\lambda_{n}) \in \prod^n_{i=1} \Lambda \mid s(\lambda_i) = r(\lambda_{i+1})\} &\text{ otherwise.}
    \end{cases}
\]
For $n \ge 2$, and for $i \in \{0, \ldots, n\}$  define a map $d_i : \Lambda^*_n \to
\Lambda^*_{n-1}$ by
\[
d_i(\lambda_1,\ldots,\lambda_{n} ) :=
\begin{cases}
 (\lambda_2,\ldots,\lambda_{n}) & \text{if  } i = 0,\\
 (\lambda_1, \lambda_2,\ldots, \lambda_i \lambda_{i+1},\ldots, \lambda_{n}) & \text{if  } 0 < i < n,\\
  (\lambda_1, \ldots, \lambda_{n-1}) & \text{if  } i = n;
\end{cases}
\]

\begin{definition} (cf.\  \cite[Remark 2.14]{gillaspy-wu}) Let  $\Lambda$  be a $k$-graph.
For $n \ge 0$, let $C_n(\Lambda) := \Z \Lambda^{*\,n}$, the free abelian group with
generators indexed by $\Lambda^{*\,n}$. Identifying elements of $\Lambda^{*\,n}$ with the
corresponding generators of $C_n(\Lambda)$, we regard the boundary maps $d_i$ as
homomorphisms $d_i : C_n(\Lambda) \to C_{n-1}(\Lambda)$. We obtain homomorphisms
$\partial_n : C_n(\Lambda) \to C_{n-1}(\Lambda)$ by $\partial_n := \sum^n_{i=1} (-1)^{i}
d_i$ for $n \ge 2$. Regarding $s, r$ as maps from $C_1(\Lambda)$ to $C_0(\Lambda)$ define
$\partial_1 : C_1(\Lambda) \to C_0(\Lambda)$ by $\partial_1 := s - r$ and   $\partial_0 :
C_0(\Lambda) \to \{0\}$ to be the zero map. Standard calculations show that
$(C_*(\Lambda), \partial_*)$ is a complex. The resulting groups $H_n(\Lambda) :=
\ker(\partial_n)/\operatorname{image}(\partial_{n+1})$ are called the \textit{categorical
homology groups} of $\Lambda$.
\end{definition}

Note that one may define a more general homology theory with coefficients in an abelian group $A$ as in
\cite{gillaspy-wu} but we do not need this level of generality.

\begin{example} \label{ex:1graphom}
Let $\Lambda$ be the $1$-graph with vertex connectivity matrix $\left(
\begin{smallmatrix} 5 & 2 \\ 2 & 3\end{smallmatrix} \right)$. By
\cite[Theorem~6.2]{gillaspy-wu} the categorical homology of $\Lambda$ coincides with its
cubical homology, which can computed as follows. Since $\Lambda$ is connected we have
$H_0 ( \Lambda ) \cong \Z$; since $\Lambda$ is finite its first homology group is free
abelian with rank equal to its Betti number $p = |\Lambda^1| - |\Lambda^0| + 1 = 11$.
Hence $H_1(\Lambda ) \cong \Z^{11}$.  Moreover, $H_n(\Lambda ) = 0$ for all $n > 1$. We
will return to this example in Remark~\ref{rmk:nolink} after establishing how to compute
the homology of the associated groupoid.
\end{example}

\begin{remark}
One can check that $C_0(\Lambda)$ together with the subgroups of the $C_n(\Lambda)$ for
$n \ge 1$ generated by elements $(\lambda_1, \cdots, \lambda_n)$ in which each $\lambda_i
\not \in \Lambda^0$ form a subcomplex under the same boundary maps $\partial_n$, and that
the  homology of this subcomplex is isomorphic to $H_*(\Lambda)$.
\end{remark}

Recall that if $\Lambda$ is a $k$-graph and $\lambda,\mu \in \Lambda$ satisfy $s(\lambda) =
s(\mu)$, then the cylinder set $Z(\lambda,\mu)$ is defined as $Z(\lambda,\mu) =
\{(\lambda x, d(\lambda) - d(\mu), \mu x) : x \in s(\lambda)\Lambda^\infty\} \subseteq \G_\Lambda$, and is a
compact open subset of $\G_\Lambda$.

\begin{definition}
Let $\Lambda$ be a $k$-graph. Let $\G_\Lambda$ be the associated groupoid
(see~\eqref{eq:dr-kgraph}) and for each $n$ let $\G_\Lambda^{(n)}$ be the collection of
composable $n$-tuples in $\G_\Lambda$ as in~\eqref{eq:gndef}. For $(\lambda_1, \dots,
\lambda_n) \in \Lambda^{*\,n} \subseteq C_n(\Lambda)$, define
\[\textstyle
Y(\lambda_1, \dots, \lambda_n) := \big(\prod^n_{i=1} Z(\lambda_i, s(\lambda_i))\big) \cap \G_\Lambda^{(n)}.
\]

Let $\Psi_* : C_*(\Lambda) \to C_c(\G_\Lambda^{(*)}, \Z)$ be the homomorphism such that
$\Psi_0(v) = 1_{Z(v)} \in C_c(\G_\Lambda^{(0)}, \Z)$ for $v \in \Lambda^0$, and
$\Psi_n(\lambda_1, \dots, \lambda_n) = 1_{Y(\lambda_1, \dots, \lambda_n)}$ for $n \ge 1$
and $(\lambda_1, \dots, \lambda_n) \in \Lambda^{*\,n}$.
\end{definition}

\begin{theorem}
\label{them-commuting-w-boundaries} Let $\Lambda $ be a $k$-graph.  The maps $\Psi_*:
{C}_*(\Lambda) \to C_c(\G_\Lambda^{(*)}, \Z)$ defined above comprise a chain map, and
induce a homomorphism $\Psi_* : H_*(\Lambda) \to H_*(\G_\Lambda)$.
\end{theorem}
\begin{proof}
It suffices to prove that $\Psi_*$ intertwines the boundary maps on generators of
${C}_*(\Lambda)$. Fix $\lambda \in \Lambda = \Lambda^{*\,1} \subseteq {C}_1(\Lambda)$.
Then
\begin{align*}
\partial_1(\Psi_1(\lambda)) &= \partial_1(1_{Y(\lambda)}) =  s_*(1_{Y(\lambda)}) - r_*(1_{Y(\lambda)})
     = 1_{Z(s(\lambda))} - 1_{Z(r(\lambda))}\\
     &= \Psi_0(s(\lambda)) - \Psi_0(r(\lambda)) = \Psi_0(\partial_1(\lambda)).
\end{align*}
For $n \ge 2$, it suffices to prove that given an element
$(\lambda_1, \dots, \lambda_n) \in \Lambda^{*\,n}$ and any $0 \le i \le n$,
we have
\begin{equation}\label{eq:commutation}
(d_i)_*(\Psi_n(\lambda_1, \dots, \lambda_n))
    = \Psi_n(d_i(\lambda_1, \dots, \lambda_n)).
\end{equation}
In the following calculation, given sets $Z_1, \dots, Z_n \subseteq \G_\Lambda$ we define
 \[
 Z_1 * Z_2 * \cdots * Z_n := (Z_1 \times Z_2 \times \cdots \times Z_n) \cap \G_\Lambda^{(n)}.
 \]
Note that $Y(\lambda_1, \dots, \lambda_n) = Z(\lambda_1, s(\lambda_1)) * Z(\lambda_2,
s(\lambda_2)) * \cdots * Z(\lambda_n, s(\lambda_n))$. First suppose that $1 \le i \le
n-1$. Then
\begin{align*}
(d_i)_*(\Psi_n(\lambda_1, \dots, \lambda_n))
    &=  (d_i)_*\big(1_{Y(\lambda_1, \dots, \lambda_n)}\big)
      = (d_i)_*\big(1_{Z(\lambda_1, s(\lambda_1)) * Z(\lambda_2, s(\lambda_2)) * \cdots * Z(\lambda_n, s(\lambda_n)})\big)\\
    &= 1_{Z(\lambda_1, s(\lambda_1)) * \cdots * Z(\lambda_i, s(\lambda_i))Z(\lambda_{i+1}, s(\lambda_{i+1})) *\cdots * Z(\lambda_n, s(\lambda_n))},
\end{align*}
and
\begin{align*}
\Psi_n(d_i(\lambda_1, \dots, \lambda_n))
    &= \Psi_n(\lambda_1, \dots, \lambda_i\lambda_{i+1}, \dots, \lambda_n) =  1_{Y(\lambda_1, \dots, \lambda_i\lambda_{i+1}, \dots, \lambda_n)}\\
    &= 1_{Z(\lambda_1, s(\lambda_1)) * \cdots * Z(\lambda_i\lambda_{i+1}, s(\lambda_{i+1})) *\cdots * Z(\lambda_n, s(\lambda_n))}.
\end{align*}
Since the cylinder sets $Z(\lambda_i, s(\lambda_i))$ and $Z(\lambda_{i+1},
s(\lambda_{i+1}))$ are both bisections and since $s(Z(\lambda_i, s(\lambda_i))) =
Z(s(\lambda_i)) = r(Z(\lambda_{i+1}, s(\lambda_{i+1})))$, we have $Z(\lambda_i,
s(\lambda_i))Z(\lambda_{i+1}, s(\lambda_{i+1})) = Z(\lambda_i\lambda_{i+1},
s(\lambda_{i+1}))$, and~\eqref{eq:commutation} follows.

Now consider $i = 0$ (the case $i = n$ is very similar).  We have
\begin{align*}
(d_0)_*(\Psi_n(\lambda_1, \dots, \lambda_n))
    &= (d_0)_*\big(1_{Y(\lambda_1, \dots, \lambda_n)}\big)\\
    &= (d_0)_*\big(Z(\lambda_1, s(\lambda_1)) * Z(\lambda_2, s(\lambda_2)) * \cdots * Z(\lambda_n, s(\lambda_n))\big)\\
    &= 1_{Z(\lambda_2, s(\lambda_2)) * \cdots * Z(\lambda_n, s(\lambda_n))} \\
    &= 1_{Y(\lambda_2, \cdots, \lambda_n)}\\
    &= \Psi_n(\lambda_2, \cdots, \lambda_n)\\
    &= \Psi_n(d_0(\lambda_1, \dots, \lambda_n)). \qedhere
\end{align*}
\end{proof}

In general the map $\Psi_*$ is neither injective nor surjective: see
Remark~\ref{rmk:nolink}.

\subsection{The HK conjecture for 1-graph groupoids and 2-graph groupoids}

In this subsection we apply the results of \S \ref{sec:DR groupoids} to groupoids associated to
$1$-graphs and $2$-graphs.

Recall from \S \ref{s:definitions} that if $\Lambda$ is a $k$-graph, then there is an
action $\sigma$ of $\N^k$ by endomorphisms on its infinite-path space $\Lambda^\infty$
and that the $k$-graph $C^*$-algebra coincides with the $C^*$-algebra
$C^*(\G(\Lambda^\infty, \sigma))$ of the associated Deaconu--Renault groupoid (see
\cite{KP2000}). We begin this section by showing that the homology of $\G(\Lambda^\infty,
\sigma)$ as computed in Theorem~\ref{thm:main H* computation} coincides with the homology
of the complex $D^\Lambda_*$ used by Evans \cite{evans} to compute the $K$-theory of
$C^*(\Lambda)$.

The complex $(D_*^\Lambda, \partial_*)$ is given as follows. We write $\varepsilon_1,
\dots , \varepsilon_k$ for the generators of $\Z^k$, and we write $\varepsilon_v$ for the
generators of $\Z{\Lambda^0}$. We write $M_j \in M_{\Lambda^0}(\Z)$ for the vertex matrix
given by $M_j(v,w) = |v \Lambda^{e_j} w|$, which we regard as an endomorphism of
$\Z{\Lambda^0}$. For $p \ge 0$ we define $D^\Lambda_p = \bigwedge^p \Z^k \otimes
\Z{\Lambda^0}$ and we define $D^\Lambda_0 = \Z \Lambda^0$, and $D^\Lambda_p = \{0\}$ for
$p > k$. For $p \ge 1$ we define $\partial_p : D^\Lambda_p \to D^\Lambda_{p-1}$ by
$\partial_p = 0$ if $p > k$,
\[
\partial_p(\varepsilon_{i_1} \wedge \dots \wedge \varepsilon_{i_p} \otimes \varepsilon_v)
    = \sum_{j=0}^p (-1)^{j+1} \varepsilon_{i_1} \wedge \dots \wedge \widehat{\varepsilon}_{i_j}
        \wedge \dots \wedge \varepsilon_{i_p} \otimes (I - M^t_{i_j})\varepsilon_v
\]
if $2 \le p \le k$ and
\[
\partial_1(\varepsilon_i \otimes \varepsilon_v) = (I - M^t_i)\varepsilon_v.
\]

In the following proposition, we establish an isomorphism between the homology of Evans'
complex $D^\Lambda_*$ and the homology of the complex $A^\sigma_*$ associated to the
shift maps $\sigma^n$ on $\Lambda^\infty$. We could obtain the existence of isomorphisms
$H_*(D^\Lambda_*) \cong H_*(A^\sigma_*)$ using that, by Evans' results, $H_*(D^\Lambda_*)
\cong H_*(\Z^k, K_0(C^*(\Lambda \times_d \Z^k)))$, that by Matui's results,
$H_*(A^\sigma_*) \cong H_*(\Z^k, H_0(\G_\Lambda \times_c \Z^k))$, and then by identifying
$C^*(\Lambda \times_d \Z^k)$ with $C^*(\G_\Lambda \times_c \Z^k)$ and applying the HK
conjecture for AF groupoids as stated in Corollary~\ref{cor:lim homology}. However, we
have chosen to present a more direct proof, which also has the advantage that it shows
that the natural inclusion $\Z\Lambda^0 \hookrightarrow C_c(\Lambda^\infty, \Z)$ induces
the isomorphism.

\begin{proposition}\label{prp:homology for k-graphs}
Let $\Lambda$ be a row-finite $k$-graph with no sources. Let $\G_\Lambda =
\G(\Lambda^\infty, \sigma)$ be the associated groupoid, and let $(A^\sigma_*)$ be the
complex of Theorem~\ref{thm:main H* computation}. Then the homomorphism $\iota :
\Z\Lambda^0 \to C_c(\Lambda^\infty, \Z)$ determined by $\iota(1_v) = 1_{Z(v)}$ induces an
isomorphism $H_*(A^\sigma_*) \to H_*(D^\Lambda_*)$. In particular,
$H_*(\G(\Lambda^\infty, \sigma)) \cong H_*(D^\Lambda_*)$.
\end{proposition}

\begin{proof}
Corollary~\ref{cor:lim homology} shows that we can express the complex $\bigwedge^* \Z^k
\otimes H_0(c^{-1}(0))$ as the direct limit of the $A^\sigma_*$ under the maps induced by
the $\sigma_*$. Lemma~\ref{lem:homology with limits} shows that the inclusion of
$A^\sigma_*$ in $\bigwedge^* \Z^k \otimes H_0(c^{-1}(0))$ induces an isomorphism in
homology. Similarly, in \cite[Theorem 3.14]{evans} Evans proves that one can express the complex $\bigwedge^* \Z^k
\otimes K_0(C^*(\Lambda \times_d \Z^k))$ as the direct limit of $D^\Lambda_*$ under the
maps induced by the $M^t_*$, and the inclusion of $D^\Lambda_*$ that takes $1_v$ to the class
of $p_{v,0}$ in the direct limit induces an isomorphism in homology. By
\cite[Theorem 5.2]{KP2000}, there is an isomorphism $C^*(\Lambda \times_d \Z^k) \cong
C^*(\G(\Lambda^\infty, \sigma) \times_c \Z^k)$ that carries $p_{(v,0)}$ to $1_{Z(v)
\times \{0\}}$. So Corollary~\ref{cor:lim homology} shows that there is an isomorphism
$K_0(C^*(\Lambda \times_d \Z^k)) \cong H_0(\G(\Lambda^\infty, \sigma) \times_c \Z^k)$
that takes $[p_{(v,0)}]$ to $[1_{Z(v) \times \{0\}}]$. Lemma~\ref{lem:skew gpd homology}
therefore implies that there is an isomorphism $K_0(C^*(\Lambda \times_d \Z^k)) \cong
H_0(c^{-1}(0))$ induced by the map that carries the class of $p_{v,0}$ to $1_{Z(v)}$.

Given $v \in \Lambda^0$ and $n \in \N^k$, we have
\begin{equation}\label{eq:sigma<->M}
\sigma^n_*(\iota(1_v)) = \sigma^n_*(1_{Z(v)})
    = \sum_{\lambda \in v\Lambda^n} \sigma^n_*(1_{Z(\lambda)})
    = \sum_{\lambda \in v\Lambda^n} 1_{Z(s(\lambda))}
    = \iota(M^t_n(1_v)).
\end{equation}
Hence $\iota$ induces a map of complexes $\iota_* : D^\Lambda_* \to A^\sigma_*$ that
intertwines the $\sigma^n_*$ with the $(M_n^t)_*$. The same computation combined with the
universal property of the direct limit shows that we obtain the following commuting
diagram.
\[
\begin{tikzpicture}
    \node (t1) at (0, 2) {$D^\Lambda_*$};
    \node (t2) at (3, 2) {$D^\Lambda_*$};
    \node (t3) at (6, 2) {$\dots$};
    \node (t4) at (9, 2) {$\bigwedge^* \Z^k \otimes H_0(c^{-1}(0))$};
    \node (b1) at (0, 0) {$A^\sigma_*$};
    \node (b2) at (3, 0) {$A^\sigma_*$};
    \node (b3) at (6, 0) {$\dots$};
    \node (b4) at (9, 0) {$\bigwedge^* \Z^k \otimes H_0(c^{-1}(0))$};
    \draw[->] (t1)--(b1) node[pos=0.5, right] {$\iota_*$};
    \draw[->] (t2)--(b2) node[pos=0.5, right] {$\iota_*$};
    \draw[->] (t4)--(b4) node[pos=0.5, right] {$\iota_*^\infty$};
    \draw[->](t1)--(t2) node[pos=0.5, above] {$M^t_*$};
    \draw[->](b1)--(b2) node[pos=0.5, above]{$\sigma^n_*$};
    \draw[->](t2)--(t3);
    \draw[->](b2)--(b3);
    \draw[->](t3)--(t4);
    \draw[->](b3)--(b4);
\end{tikzpicture}
\]
By definition, $\iota_*([1_{v}]) = [1_{Z(v)}]$ for every $v \in \Lambda^0$. Write $(
M^t_{n,\infty} )_*$ for the map from the $n$th copy of $D^\Lambda_*$ into $\bigwedge^*
\Z^k \otimes  H_* ( c^{-1} (0) )$ and write $\sigma^{n,\infty}_*$ for the map from the
$n$th copy of $A^\sigma_*$ to $\bigwedge^* \Z^k \otimes  H_* ( c^{-1} (0) )$. Then $(
M^t_{n,\infty} )_* ( 1_v ) = [ 1_{Z(\mu)} ]$ for any $\mu \in \Lambda^n v$. Since
\[
\sigma^{n,\infty}_* ( \iota_* ( 1_v ) ) = \sigma^{n,\infty}_* ( 1_{Z(v)} ) = \sigma^{n,\infty}_* ( \sigma^n_* ( 1_{Z(\mu)} ) )
= \sigma^{n,\infty}_* ( 1_{Z(\mu)} ) = [ 1_{Z(\mu)} ],
\]
we see by commutativity of the diagram that $\iota^\infty_* ( 1_{Z(\mu)} ) = [ 1_{Z(\mu)}
]$ for all $\mu$. Since the $1_{Z(\mu)}$ generate $C_c(\Lambda^\infty, \Z)$, we deduce
that $\iota^\infty_*$ is the identity map, and therefore induces the identity map in
homology. Since the maps $H_*(D^\Lambda_*) \to H_*(\bigwedge^* \Z^k \otimes
H_0(c^{-1}(0)))$ and $H_*(A^\sigma_*) \to H_*(\bigwedge^* \Z^k \otimes H_0(c^{-1}(0)))$
are isomorphisms, and the diagram above commutes, the functoriality of homology implies
that $\iota_*$ induces an isomorphism $H_*(D^\Lambda_*) \to H_*(A^\sigma_*)$.

The final statement follows from Theorem~\ref{thm:main H* computation}.
\end{proof}

Though we already know that graph groupoids belong to $\M$ by Remark~\ref{rmk:1graphHK},
the following result goes a step further, computing the homology of the $1$-graph and
$2$-graph groupoids in terms of the vertex matrices of the $1$-graph or $2$-graph. Recall
that given a $k$-graph $\Lambda$ and given $i \le k$ we write $M_i$ for the $\Lambda^0
\times \Lambda^0$ integer matrix given by $M_i(v,w) = |v \Lambda^{e_i} w|$. If $\Lambda$
is the path category of a directed graph $E$, then $M_1$ is just the usual adjacency
matrix $A_E$ of $E$.

\begin{corollary}[{see \cite[Proposition~3.16]{evans}}]\label{cor:HK for graphs}
\begin{enumerate}
\item Let $E$ be a row-finite graph with no sources. Then
\begin{align*}
K_0 ( C^* (E) ) \cong H_0(\G_E) &\cong \coker(I - A^t_E)\quad\text{ and }\\
K_1 ( C^* (E) ) \cong H_1(\G_E)  &\cong \ker(I - A^t_E).
\end{align*}
\item Let $\Lambda$ be a row-finite $2$-graph with no sources. Then
\begin{align*}
K_0(C^*(\Lambda)) \cong H_0(\G_\Lambda) &\cong \Z\Lambda^0/\big((I - M_1^t, I-M_2^t )(\Z\Lambda^0)^2\big)
\oplus \ker(M_2^t - I) \cap \ker(I - M_1^t),\\
\intertext{and}
K_1(C^*(\Lambda)) \cong H_1(\G_\Lambda) &\cong
\ker(I - M_1^t, I - M_2^t )/\big(\begin{smallmatrix}(M_2^t - I) \\ (I - M_1^t)\end{smallmatrix}\big)\Z\Lambda^0.
\end{align*}
\end{enumerate}
In particular, graph groupoids and $2$-graph groupoids belong to $\M$.
\end{corollary}

\begin{proof}
Theorems \ref{thm:DR k=1}~and~\ref{thm:DR k=2} establish the isomorphisms between
homology and $K$-theory. The descriptions of the homology groups follow from
Proposition~\ref{prp:homology for k-graphs} and the definition of the complex
$D^\Lambda_*$.
\end{proof}

\begin{remark} \label{rmk:nolink}
The strongly connected $1$-graph $\Lambda$ described in Example~\ref{ex:1graphom} has homology given by
\[
H_n(\Lambda) =
\begin{cases}
 \Z & \text{if  } n = 0, \\
 \Z^{11} & \text{if  } n = 1, \\
  0 & \text{otherwise.}
\end{cases}
\]
By Corollary~\ref{cor:HK for graphs} the homology of $\G_\Lambda$ is
\[
H_n(\G_\Lambda) =
\begin{cases}
\Z/2\Z \oplus \Z / 2\Z & \text{if $n = 0$,} \\
  0 & \text{otherwise.}
\end{cases}
\]
So the map  $\Psi_0 : H_0(\Lambda) \to H_0(\G_\Lambda)$ of
Theorem~\ref{them-commuting-w-boundaries} is neither surjective nor injective.
\end{remark}

\begin{remark}
In \cite[Proposition~3.18]{evans}, Evans shows that if $\Lambda$ is a $3$-graph and $\{(I
- M_i^t) a : 1 \le i \le 3, a \in \Z\Lambda^0\}$ generates $\Z\Lambda^0$, then the
$K$-theory of $C^*(\Lambda)$ is equal to the homology of $D^\Lambda$. So the groupoids of
such $k$-graphs belong to $\M$.

We also see that if $\Lambda$ is a 3-graph or a 4-graph for which the page 3
differentials in Kasparov's sequence are zero and $H_2(D^\Lambda)$ (and $H_3(D^\Lambda)$
in the case of a 4-graph) are free abelian, then $K_*(C^*(\Lambda))$ is determined by
$H_*(D^\Lambda)$ and so $\G_\Lambda$ belongs to $\M$; but of course the hypothesis on the
differential maps in Kasparov's sequence are not checkable in practice.
\end{remark}

\begin{remark}
Suppose that $\Lambda$ is a finite 3-graph. Then the homology groups $H_*(D^\Lambda_*)$
are finite rank, and $H_3(D^\Lambda_*)$ is free abelian. Consequently, if it were
possible to construct an example of a finite 3-graph for which the page 3 differential
$d^3_{3,0}$ in Kasparov's spectral sequence was nontrivial, then consideration of the
ranks of the groups involved would show that the associated groupoid did not satisfy the
HK conjecture, even up to stabilisation by $\Q$.
\end{remark}

\subsection{One vertex \texorpdfstring{$k$}{k}-graphs}

In this section we will show that if $\Lambda$ is a $1$-vertex $k$-graph in which each
$|\Lambda^{e_i}| \ge 2$ and in which $\gcd(|\Lambda^{e_1}|-1, \dots , |\Lambda^{e_k}|-1)
= 1$, then $K_*(C^*(\Lambda)) = H_*( \G_\Lambda) = 0$. We work with row-finite $k$-graphs
throughout, but we include a comment at the end of the section indicating how to extend
our $K$-theory calculation to non-row-finite $k$-graphs. A similar result has been proved
in \cite[Theorem 6.4(a)]{baomsta} under the assumption that the elements of $\{
|\Lambda^{e_i}| : 1 \le i \le k \}$ are pairwise relatively prime.

The key point is the following consequence of Matui's K\"unneth formula for the groupoid
homology of an ample Hausdorff groupoid (see \cite[Theorem 2.4]{Matui-2016}).

\begin{theorem}\label{thm:H computation}
Let $\Lambda$ be a row-finite single-vertex $k$-graph with at least two edges of each colour,
and write $N_i := |\Lambda^{e_i}| - 1$ for each $i \le k$. Then
\[
H_n(\G_\Lambda) \cong \begin{cases}
    (\Z_{\gcd(N_1, \dots, N_k)})^{\binom{k-1}{n}} &\text{ if $0 \le n \le k-1$}\\
    0 &\text{ otherwise.}
    \end{cases}
\]
\end{theorem}
\begin{proof}
We proceed by induction. This follows from Corollary~\ref{cor:HK for graphs}(1) if $k = 1$.
Suppose it holds for $k = K-1$ and that $\Lambda$ is a $K$-graph with one vertex and at least two edges of each colour.
Since the complex $D^\Lambda_*$ is independent of the factorisation rules in $\Lambda$,
Proposition~\ref{prp:homology for k-graphs} shows that $H_*(\G_\Lambda)$ is independent
of the factorisation rules. So we can assume that $\Lambda = B_{N_1 + 1} \times \dots
\times B_{N_k + 1}$ and so $\G_\Lambda = \prod^k_{i=1} \G_{B_{N_i + 1}}$.

Write $\G = \prod^{K-1}_{i=1} \G_{B_{N_i + 1}}$ and $\mathcal{H} = \G_{B_{N_K + 1}}$ so
that $\G_\Lambda \cong \G \times \mathcal{H}$. Matui's K\"unneth theorem gives a split
exact sequence
\[
0 \to \bigoplus_{i+j= n} H_i(\G) \otimes H_j(\mathcal{H})
     \longrightarrow H_n(\G \times \mathcal{H})
     \longrightarrow \bigoplus_{i+j= n-1} \operatorname{Tor}(H_i(\G), H_j(\mathcal{H}))
     \to 0,
\]
and since $H_*(\mathcal{H}) = (\Z_{N_K}, 0, 0, \dots)$ this collapses to a split exact
sequence
\[
0 \to H_n(\G) \otimes \Z_{N_K}
     \longrightarrow H_n(\G \times \mathcal{H})
     \longrightarrow \operatorname{Tor}(H_{n-1}(\G), \Z_{N_K})
     \to 0.
\]
Since the sequence splits, the middle term is the direct sum of the two ends. Write
$\gamma := \gcd(N_1, \dots, N_{K-1})$. Then the inductive hypothesis gives $H_n(\G) =
\Z_\gamma^{^{\binom{K-2}{n}}}$ and $H_{n-1}(\G) = \Z_\gamma^{\binom{K-2}{n-1}}$. Also,
$\gcd(\gamma, N_K) = \gcd(N_1, \dots, N_K)$. As $\otimes$ and $\operatorname{Tor}$ are
both additive in the first variable and $\Z_l \otimes \Z_m = \Z_{\gcd(l,m)} =
\operatorname{Tor}(\Z_l, \Z_m)$, we have
\begin{align*}
H_n(\G \times \mathcal{H})
    &= (\Z_{\gcd(\gamma, N_K)})^{\binom{K-2}{n}} \oplus (\Z_{\gcd(\gamma, N_K)})^{\binom{K-2}{n-1}}\\
    &= (\Z_{\gcd(N_1, \dots, N_K)})^{\binom{K-2}{n} + \binom{K-2}{n-1}}\\
    &= (\Z_{\gcd(N_1, \dots, N_K)})^{\binom{K-1}{n}}.\qedhere
\end{align*}
\end{proof}

We deduce that the groupoids of $1$-vertex $k$-graphs in which there are at least two
edges of each colour, and in which $\gcd(|\Lambda^{e_1}|-1, \dots , |\Lambda^{e_k}|-1) =
1$ belong to $\M$.

\begin{corollary}
If $\Lambda$ is a single-vertex $k$-graph with at least two edges of each colour, and
$\gcd(|\Lambda^{e_1}|-1, \dots , |\Lambda^{e_k}|-1) = 1$, then $K_*(C^*(\Lambda)) =
H_*(\G_\Lambda) = 0$. In particular, if $C^*(\Lambda)$ is simple then it is isomorphic to
$\mathcal{O}_2$.
\end{corollary}

\begin{proof}
Theorem~\ref{thm:H computation} shows that $H_*(\G_\Lambda) = 0$. It follows that the
groups $E^2_{p,q}$ in Matui's spectral sequence are all zero. Since the terms $F^2_{p,q}$
in Kasparov's spectral sequence \cite{kasparov} (see also \cite{evans}) are given by
$E^2_{p,0}$ if $q$ is even, and 0 if $q$ is odd, we deduce that the $F^2_{p,q}$ are all
zero. So Evans' spectral sequence collapses, and we obtain $K_*(C^*(\Lambda)) = 0$ as
well. The final statement follows from \cite[Proposition 8.8 and  Corollary
8.15]{Sims:CJM06}.
\end{proof}

\begin{remark}
Unfortunately, if $\gcd(N_1, \dots, N_k) > 1$, we can conclude little new about the HK
conjecture. The problem is that in $K$-theory, with the exception of tensor products, we
only obtain from Evans' spectral sequence that the $K$-groups have filtrations of length
at most $k-1$ with subquotients equal to direct sums of copies of $\Z_{\gcd(N_1, \dots,
N_k)}$.
\end{remark}

\begin{remark}\label{rmk:non-row-finite}
The above discussion deals only with row-finite $k$-graphs. We can extend the $K$-theory
calculation to non-row-finite examples as follows. If $\Lambda$ is any 1-vertex
$k$-graph, and $\Gamma$ is a $1$-vertex $(k+1)$-graph such that $d_{\Gamma}^{-1}(\N^k)
\cong \Lambda$ and $\Gamma^{e_{k+1}}$ is infinite, then as in \cite{BCSS} we can make the
identification
\[
C^*(\Gamma)
    \cong \mathcal{T}_{\ell^2(C^*(\Lambda))_{C^*(\Lambda)}},
\]
and deduce from Pimsner's \cite[Theorem~4.4]{Pimsner} that the inclusion $C^*(\Lambda)
\hookrightarrow C^*(\Gamma)$ determines a $KK$-equivalence, so $K_*(C^*(\Gamma)) \cong
K_*(C^*(\Lambda))$.

Applying this iteratively, we can compute the $K$-theory of the $C^*$-algebra of a
$1$-vertex, not-necessarily-row-finite $k$-graph as the $K$-theory of the $C^*$-algebra
of the subgraph consisting only of those coordinates in which there are finitely many
edges.
\end{remark}

\end{document}